\documentclass[11pt,reqno]{amsart}
\usepackage{amssymb,mathrsfs,graphicx,subfigure, enumerate}
\usepackage{amsmath,amsfonts,amssymb,amscd,amsthm,bbm}
\usepackage{graphicx,colortbl,here}
\usepackage{extpfeil}
\numberwithin{figure}{section}
\usepackage{kotex}
\usepackage[utf8]{inputenc}
\usepackage[T1]{fontenc}
\usepackage{bm}

\def\d {{\partial}}

\def \rpsi_i {|\psi_i \rangle}
\def \lpsi_i {\langle \psi_i|}
\def \lrpsi_i{\langle \psi_i | \psi_i \rangle}
\def \rpsi_k {|\psi_k \rangle}
\def \lpsi_k {\langle \psi_k|}
\def \lrpsi_k{\langle \psi_k | \psi_k \rangle}

\newcommand{\bbr}{\mathbb R}

\def\d{\mathrm{d}}

\newcommand{\dm}{d}

\DeclareMathOperator*{\argmin}{argmin}

\newcommand{\ba}{\begin{aligned}}
\newcommand{\ea}{\end{aligned}}

\newcommand{\be}{\begin{equation}}
\newcommand{\ee}{\end{equation}}

\renewcommand{\d}{  {\textup{d}} }
\newcommand{\dt} {    {  \textup{d}t   }    }

\newtheorem{theorem}{Theorem}[section]
\newtheorem{lemma}{Lemma}[section]
\newtheorem{corollary}{Corollary}[section]
\newtheorem{proposition}{Proposition}[section]
\newtheorem{remark}{Remark}[section]
\newtheorem{definition}{Definition}[section]

\begin{document}

\title[]{Finite-dimensional reduction of a Wasserstein gradient flow and sharp decay rates}

\author[D. Kim]{Dohyun Kim}
\address[D. Kim]{\newline Department of Mathematics Education and Institute of Pure and Applied Mathematics, \newline Sungkyunkwan University, Seoul 03063, Republic of Korea}
\email{dohyunkim@skku.edu}

\author[H. Park]{Hansol Park}
\address[H. Park]{\newline Department of Mathematics, \newline National Tsing Hua University, Hsinchu 30013, Taiwan}
\email{hansolpark@math.nthu.edu.tw}

\author[W. Shim]{Woojoo Shim}
\address[W. Shim]{\newline Department of Mathematics Education, \newline Kyungpook National University, Daegu 41566, Republic of Korea }
\email{wjshim@knu.ac.kr}

\thanks{\textbf{Acknowledgment.}
The work of D. Kim was supported by the National Research Foundation of Korea (NRF) grant funded by the Korean government (MSIT) (RS-2024-00454452), and the work of H. Park was supported by the National Science and Technology Council (NSTC), Taiwan (Grant No. NSTC 115-2115-M-007-001-MY3).}

\begin{abstract}
We study the Wasserstein gradient flow generated by a family of extended generalized variance functionals, defined as the expected squared $n$-dimensional volume of a simplex, which includes the classical variance-type interaction and generalized variance as special cases. The key structural observation is that this functional depends only on the covariance matrix. Consequently, the Wasserstein gradient flow reduces to a finite-dimensional system for the eigenvalues of the covariance matrix, and the full measure-valued solution can be recovered through an explicit linear pushforward representation. Using this representation, we establish global well-posedness for arbitrary initial data in $\mathcal P_2(\mathbb R^d)$ without assuming compact support. We also study the long-time behavior of the flow. For every initial datum in $\mathcal P_2(\mathbb R^d)$, the solution converges to a limiting equilibrium measure whose covariance has rank strictly less than $n$. Moreover, we obtain sharp convergence rates in all spectral regimes: exponential in the non-degenerate case and algebraic with the optimal exponent in the degenerate case.
\end{abstract}


\makeatletter
\@namedef{subjclassname@2020}{%
  \textup{2020} Mathematics Subject Classification}
\makeatother


%49Q22: optimal transport
%35A01: Existence problems for PDEs: 
%35B40: Asymptotic behavior of solutions to PDEs
%35R06: PDEs with measure

\date{\today}

\maketitle

%\tableofcontents

\section{Introduction}
\setcounter{equation}{0}

Gradient flows in the Wasserstein space provide a powerful framework for studying the dissipative evolution of probability measures. By starting from a functional defined on  a space of probability measures, the associated Wasserstein gradient flow describes the steepest descent dynamics with respect to the quadratic distance. This viewpoint has become a central tool in the analysis of variational problems,
interacting particle systems, and dissipative dynamics. One of the celebrated examples is the variational formulation of the Fokker--Planck equation through the Jordan--Kinderlehrer--Otto scheme \cite{JKO} (JKO scheme for brevity). In many models, a main analytical challenge is to rigorously  identify the gradient flow structure, establish global or local well-posedness in a suitable sense and determine the asymptotic behavior of solutions (see Section \ref{sec:2.1} for more details on the Wasserstein gradient flows). \newline

In this work, we investigate a Wasserstein gradient flow generated by a higher-order volume functional. Given $n+1$ points in $\bbr^d$, say, $x_0,\cdots,x_n \in \bbr^d$, we define the $n$-dimensional volume of the simplex spanned by these points.
\begin{definition}[$n$-dimensional volume of a simplex] \label{volume}
	For $x_0,\dots,x_n\in\mathbb R^d$, set
	\[
	B(x_0,\dots,x_n)
	:=
	\begin{bmatrix}
		x_1-x_0 & \cdots & x_n-x_0
	\end{bmatrix}
	\in\mathbb R^{d\times n}.
	\]
	We define the $n$-dimensional volume of the (possibly degenerate) simplex with vertices $x_0,\dots,x_n$ by
	\[
	\mathrm{Vol}_n(x_0,\dots,x_n)
	:=
	\frac{1}{n!}\sqrt{\det (B^\top B)}.
	\]
\end{definition}

 For a probability measure $\rho$ on $\bbr^d$,  we define the extended generalized variance of $\rho$:
\begin{definition}[Extended generalized variance {\cite{pronzato2017extended}}]
	For $\rho\in\mathcal P(\mathbb R^d)$, we define
	\[
	\mathcal V_n[\rho]
	:=
	\int_{(\mathbb R^d)^{n+1}}
	\mathrm{Vol}_n(x_0,\dots,x_n)^2\,\d\rho^{\otimes(n+1)}(x_0,\dots,x_n).
	\]
	We call $\mathcal V_n[\rho]$ the extended generalized variance of $\rho$.
\end{definition}
Thus, $\mathcal V_n[\rho]$ is the expected squared $n$-dimensional volume of a simplex whose vertices are sampled independently according to $\rho$. This functional can be understood as a higher-order extension of the standard variance-type interaction. For instance, when $n=1$, it reduces to the mean squared pairwise distance up to a constant. On the other hand, for $n=d$, it recovers the classical generalized variance, again up to normalization \cite{pronzato2017extended, wilks1932certain}.  For $1<n<d$, the functional interpolates between these two quantities by measuring the combined dispersion over all $n$-dimensional covariance directions. The family $\{\mathcal V_n\}_{n=1}^d$ therefore provides a dimension-indexed hierarchy of dispersion measures, ranging from the total variance to the full generalized variance. This feature becomes particularly useful when the covariance matrix is degenerate or nearly degenerate.

 In such cases, the determinant may vanish or become too coarse to distinguish the lower-dimensional structure, whereas the intermediate functionals $e_n(\Sigma)$ still detect the amount of dispersion along $n$-dimensional subspaces.  Related simplicial variance functionals have been
used to define simplicial potentials and generalized Mahalanobis distances
\cite{pronzato2018simplicial}, as well as distributional divergences and optimal-design
criteria \cite{pronzato2018bregman}. We also refer to \cite{WZ} for a recent
exterior-algebra viewpoint on generalized variances and cross-covariances.

In this work, we take a dynamical viewpoint on this family of dispersion functionals by studying their Wasserstein gradient flows. Formally, the Wasserstein gradient flow of $\mathcal V_n$ is given by the following continuity equation:
\[
\partial_t \rho_t + \nabla \cdot (\rho_t v[\rho_t]) =0,\quad v[\rho] = -\nabla \frac{\delta \mathcal V_n}{\delta \rho},
\]
where the velocity field is formally given by the negative Wasserstein gradient field of $\mathcal V_n$. At first glance, this is a nonlinear and nonlocal equation, since the velocity field is determined by the distribution itself through the first variation of the potential. One of the main novelties of this paper is that, despite its nonlocal nature, the flow has a remarkably rigid finite-dimensional structure.

\subsection{Finite-dimensional representation}
One of the key observations is the covariance representation of $\mathcal V_n$. For $\rho\in \mathcal P_2(\bbr^d)$, if we denote the mean and covariance matrix by
\begin{equation}\label{defmS}
m(\rho):=\int_{\mathbb R^d}x\,\d\rho(x),
\quad
\Sigma(\rho):=\int_{\mathbb R^d}(x-m(\rho))(x-m(\rho))^\top\,\d\rho(x),
\end{equation}
then at least formally, $\mathcal V_n$ is represented by 
\begin{equation} \label{finite}
\mathcal V_n[\rho] = \frac{n+1}{n!} e_n(\Sigma(\rho)).
\end{equation}
Here, for a symmetric matrix $A$, we denote  $e_k(A)$ by the $k$-th elementary symmetric polynomial of the eigenvalues of $A$, with the convention $e_0(A)=1$ (see Proposition \ref{prop:moment-representation}). 

By differentiating the finite-dimensional representation \eqref{finite}, the velocity field is explicitly written as 
\[
v[\rho] = -2c_n T_{n-1}(\Sigma(\rho))(x-m(\rho)),
\]
where $T_{n-1}$ is the $(n-1)$-th Newton transformation and $c_n$ is a normalization constant:
\[
T_{n-1}(A):=\sum_{k=0}^{n-1}(-1)^k e_{n-1-k}(A)A^k,\quad c_n:=\frac{n+1}{n!}
\]
It is worthwhile mentioning that the velocity field is affine in the spatial variable $x$ and depends on $\rho$ only through $m(\rho)$ and $\Sigma(\rho)$. 

This structural simplification allows us to reduce the infinite-dimensional Wasserstein gradient flow to a finite-dimensional system. Formally, the mean is conserved along the flow and the covariance matrix satisfies
\[
\frac{\d}{\d t} m_t =0,\quad \frac{\d}{\d t} \Sigma_t = -4c_n \Sigma_t T_{n-1}(\Sigma_t).
\]
 After diagonalizing the covariance matrix, the eigenvalues solve
\[
\frac{\d}{\d t} \lambda_i(t) = -4c_n \lambda_i(t) e_{n-1} (\widehat{\lambda_i(t)}), \quad 1\leq i \leq d,
\]
where $e_{n-1}(\widehat{\lambda_i(t)})$ denotes the elementary symmetric polynomial of degree $n-1$ in all eigenvalues except $\lambda_i(t)$. See Proposition  \ref{prop:covariance-dynamics}  and Corollary \ref{cor:eigenvalue-system} for the rigorous derivations. This finite-dimensional reduction is the starting point for both the well-posedness theory and the sharp asymptotic analysis.

\subsection{Global well-posedness}
Our first main result establishes global well-posedness of the Wasserstein gradient flow for arbitrary initial data in $\mathcal P_2(\bbr^d)$. More precisely, for every $\rho_0\in \mathcal P_2(\bbr^d)$, we construct a global curve $(\rho_t)_{t\geq0}$ by solving the reduced covariance system and pushing forward the initial datum through a time-dependent linear transition map. We then show that this curve is the unique $\textup{AC}^2$ weak solution of the continuity equation driven by the above explicit velocity field. In addition, it is also the unique maximal slope solution of $\mathcal V_n$ in $\mathcal P_2(\bbr^d)$ and satisfies the energy dissipation identity (EDI for brevity). To this end, we consider a solution to the following continuity equation:
\begin{equation}  \label{PDE}
	\partial_t \rho_t + \nabla \cdot ( \rho_t v_t) =0, \quad v_t(x) := v[\rho_t](x)
\end{equation} 
subject to initial data
\begin{equation}\label{initial}
\rho|_{t=0}=\rho_0.
\end{equation}
Recall that our velocity field $v[\rho_t]$ satisfies 
\begin{equation} \label{vt}
 v[\rho_t](x) = -\nabla \frac{\delta\mathcal V_n}{\delta \rho}[\rho_t] = -(n+1) \int_{(\bbr^d)^n}  \nabla_x\left(\textup{Vol}_n(x,x_1,\cdots,x_n)^2\right)\d\rho_t(x_1) \cdots \d\rho_t(x_n).
\end{equation}

We now recall the solution concepts used for the Wasserstein gradient flow.

\begin{definition} \label{def:sol}
	Let $T>0$. 
	A continuous curve $(\rho_t)\in \textup{AC}^2 (0,T;  \mathcal P_2(\bbr^d))$ is called a weak solution to the Cauchy problem for \eqref{PDE}--\eqref{initial} on $[0,T)$ if $\rho|_{t=0} = \rho_0$ and \eqref{PDE} holds in the sense of distributions: for all  $\varphi \in C_c^\infty([0,T]\times \bbr^d)$, 
	\[
	\int_0^T \int_{\bbr^d} ( \partial_t \varphi + v_t \cdot \nabla \varphi) \d\rho_t \d t + \int_{\bbr^d} \varphi(0,\cdot)\d\rho_0 = 0.
	\]
    In addition, $\rho$ is called an EDI solution of the gradient flow $\mathcal V_n$ if the following relation holds: for $t\in [0,T]$, 
		\[
		\mathcal V_n[\rho_t] + \frac12 \int_0^t |\rho'|^2(s) \d s + \frac12 \int_0^t |\partial \mathcal V_n|^2(\rho_s) \d s \leq \mathcal V_n[\rho_0].
		\]
\end{definition}

Now, we are ready to state the global well-posedness result. 

 \begin{theorem}[Finite-dimensional reduction and well-posedness]
	\label{thm:wellposed-reduction}
	Let $1\le n\le d$ and let $\rho_0\in\mathcal P_2(\mathbb R^d)$. Then the
	Wasserstein gradient flow of $\mathcal V_n$ starting from $\rho_0$ admits a unique
	maximal slope solution with respect to $|\partial \mathcal V_n|$. Moreover, this
	solution satisfies the energy dissipation identity
	\[
	\mathcal V_n[\rho_t]
	+
	\frac12\int_s^t |\rho'|^2(r)\,\d r
	+
	\frac12\int_s^t |\partial \mathcal V_n|^2(\rho_r)\,\d r
	=
	\mathcal V_n[\rho_s],
	\qquad
	0\le s\le t<\infty.
	\]
	It is also the unique $AC^2$ weak solution of
	\[
	\partial_t\rho_t+\nabla\cdot(\rho_t v[\rho_t])=0,
	\qquad
	\rho|_{t=0}=\rho_0,
	\]
	where
	\[
	v[\rho](x)
	=
	-2c_nT_{n-1}(\Sigma(\rho))(x-m(\rho)).
	\]
Finally, the flow is finite-dimensional at the level of the covariance
	matrix: if
	\[
	\Sigma(\rho_0)=P^\top D_0P,
	\qquad
	D_0=\operatorname{diag}(\lambda_1^0,\dots,\lambda_d^0),
	\]
	then
	\[
	\Sigma(\rho_t)=P^\top D_tP,
	\qquad
	D_t=\operatorname{diag}(\lambda_1(t),\dots,\lambda_d(t)),
	\]
	where
	\[
	\frac{\d}{\d t}\lambda_i(t)
	=
	-4c_n\lambda_i(t)e_{n-1}(\widehat\lambda_i(t)),
	\qquad
	1\le i\le d.
	\]
\end{theorem}

We emphasize that for the global well-posedness, no compact support assumption is imposed on the initial measure. The only assumption is the finite second moment condition, which is quite natural in the Wasserstein framework. This is possible because the affine representation of the velocity field makes it feasible to control the whole dynamics using only the mean and covariance matrix. 

Our proof consists of several ingredients. First, we show that the covariance ODE is globally well-posed, and hence determines a global transition map. Second, the corresponding pushforward curve is shown to solve the continuity equation and satisfy the maximal slope formulation. This requires, for instance, an explicit slope identity and the verification that the local slope is indeed a strong upper gradient. Finally, uniqueness is established by a $W_2$-stability estimate. Here,  we highlight that the map $\rho\mapsto   v[\rho]$ is locally Lipschitz with respect to $W_2$ on sets with bounded second moments, even without compact support (see Theorem \ref{thm:W2-stability-uniqueness}).

\subsection{Long-time behavior}

Our second main result is concerned with the long-time behavior of the Wasserstein gradient flow. We first show that since each eigenvalue of the covariance matrix is non-negative and non-increasing, the limiting covariance matrix exists. In addition, the interaction induced by $\mathcal V_n$ forces at least $d-n+1$ directions to collapse, and hence the limiting covariance matrix has rank strictly smaller than $n$.  In this regard, we can show that there exists a limiting measure $\rho_\infty \in \mathcal P_2(\bbr^d)$ such that
\[
\lim_{t\to\infty} W_2(\rho_t,\rho_\infty) =0.
\]
Thus, the Wasserstein gradient flow converges to an equilibrium measure supported on a low-dimensional affine subspace. In the generic non-degenerate case $\ell=n$, the limiting covariance matrix has rank $n-1$. In addition to the existence of an equilibrium measure, we further determine the sharp convergence rate in every spectral regime. More precisely, if $\ell=n$, then the convergence rate is exponential, whereas if $\ell<n$, then the convergence rate becomes algebraic. 

Before stating the result, we fix the asymptotic notation used below. 
For two positive functions \(f\) and \(g\) defined on \([0,\infty)\), we write
\[
f(t)\lesssim g(t)
\]
if there exist constants \(C>0\) and \(T>0\) such that
\[
f(t)\le Cg(t) \qquad \text{for all } t\ge T.
\]
We write
\[
f(t)\simeq g(t)
\]
if both \(f(t)\lesssim g(t)\) and \(g(t)\lesssim f(t)\) hold. Equivalently, 
\(f(t)\simeq g(t)\) means that there exist constants \(C>0\) and \(T>0\) such that
\[
\frac{g(t)}{C}\le f(t)\le C g(t)
\qquad \text{for all } t\ge T.
\]

\begin{theorem}[Rank collapse and sharp decay rates]
	\label{thm:long-time-rates}
	Let $(\rho_t)_{t\ge0}$ be the solution in
	Theorem~\ref{thm:wellposed-reduction}. Since $\Sigma(\rho_0)$ is a symmetric matrix, without loss of generality, we can assume that its eigenvalues are ordered as
       \[
	\lambda_1^0\ge\cdots\ge\lambda_{\ell-1}^0
	>
	\lambda_\ell^0=\cdots=\lambda_n^0
	\ge
	\lambda_{n+1}^0\ge\cdots\ge\lambda_d^0
	\]
    for some $\ell\in\{1,\dots,n\}$. Then, the following assertions hold. 
    \begin{enumerate}
    \item If \(\lambda_n^0=0\), then \(\rho_t=\rho_0\) for all \(t\ge0\).
    \item If \(\lambda_n^0>0\), there exists a limiting measure
	$\rho_\infty\in\mathcal P_2(\mathbb R^d)$ such that
	\[
	W_2(\rho_t,\rho_\infty)\to0
	\qquad\text{as }t\to\infty,
	\]
    and $\Sigma(\rho_\infty)$ has rank $\ell-1$. Moreover, if \(w_1,\dots,w_d\) are orthonormal eigenvectors of
\(\Sigma(\rho_0)\) associated with
\(\lambda_1^0,\dots,\lambda_d^0\), respectively, then
\[
\operatorname{supp}\rho_\infty
\subset
m_0+\operatorname{span}\{w_1,\dots,w_{\ell-1}\},
\qquad m_0:=m(\rho_0).
\]
Here and in the sequel, we use the convention that
\(\operatorname{span}\varnothing=\{0\}\).
    \item If \(\lambda_n^0>0\), the convergence rate of $W_2(\rho_t, \rho_\infty)$ is given as follows.
    \[
    \begin{cases}
    \displaystyle-\ln\left(W_2(\rho_t,\rho_\infty)\right)\simeq t\quad\text{if }n=\ell,\vspace{0.2cm}\\
    \displaystyle W_2(\rho_t, \rho_\infty)\simeq t^{-\frac{1}{2(n-\ell)}}\quad\text{if }n>\ell.
    \end{cases}
    \]
    %i.e., $W_2(\rho_t,\rho_\infty)$ converges to zero exponentially if $n=\ell$.
    \end{enumerate}
\end{theorem}

\begin{remark}
The rank-collapse behavior of our Wasserstein gradient flow has a natural connection with principal component analysis (PCA for brevity). Classical PCA identifies low-dimensional affine subspaces through the leading eigenvectors of the covariance matrix, or equivalently, through the directions with the largest variance. Thus, although our flow is not a PCA algorithm, its long-time dynamics exhibits a PCA-type selection of principal covariance directions. 
\end{remark}

The dichotomy between exponential and algebraic relaxations is caused by the degeneracy of the leading eigenvalues of the covariance matrix. When the $n$-th relevant eigenvalue is separated from the preceding one, the collapsing direction decays exponentially. However, when several relevant eigenvalues coincide, the reduced ODE instantaneously loses a linear spectral gap and produces an algebraic decay rate. We also demonstrate that our rates are optimal. In this sense, the long-time dynamics of the higher-order nonlocal Wasserstein
gradient flow is governed by a finite-dimensional spectral mechanism.

The present work lies at the intersection of Wasserstein gradient flows, nonlocal interaction energy functionals, and covariance-based dispersion functionals. The metric theory of Wasserstein gradient flows in probability spaces was developed in \cite{AGS2005,JKO,otto2001geometry} and related works, for instance, \cite{ambrosio2008hamiltonian,daneri2008eulerian,lisini2007characterization}. Nonlocal interaction equations have also been studied extensively, especially for pairwise interaction energies and repulsive-attractive potentials; see, for example, \cite{balague2013dimensionality,carrillo2011global,craig2017nonconvex,craig2016convergence}. Convergence rates in Wasserstein distance for related granular media equations were obtained in \cite{bolley2012uniform,carrillo2003kinetic,carrillo2006contractions} through entropy dissipation, contractivity estimates, or dissipation estimates for the Wasserstein distance. In these works, well-posedness, concentration, dimensionality of minimizers, and convergence rates are usually governed by pairwise kernels, convexity or coercivity properties, contractivity estimates, or compactness arguments.

By contrast, the energy functional considered here is an $(n+1)$-body simplex interaction. Although it is genuinely higher-order and nonlocal, it admits the exact covariance representation \eqref{finite}. This algebraic structure makes the velocity field affine in space, closes the covariance dynamics, and determines the full measure-valued flow through linear pushforward maps. Consequently, the main properties of the flow follow from an explicit finite-dimensional spectral mechanism: global well-posedness holds for arbitrary $\mathcal P_2(\mathbb R^d)$-initial data, rank collapse of the covariance matrix, and the convergence rate is sharp in each spectral regime.

The paper is organized as follows. In Section \ref{sec:2}, we recall the generalized variance functional, its covariance representation and preliminaries from the Wasserstein gradient flow framework. In Section \ref{sec:3}, we present the proof of Theorem \ref{thm:wellposed-reduction} for   global well-posedness together with the uniqueness. In Section \ref{sec:4}, we prove Theorem \ref{thm:long-time-rates} by establishing  the existence of an equilibrium measure and sharp convergence rates.

\section{Preliminaries}\label{sec:2}
\setcounter{equation}{0}

\subsection{Previous results on Wasserstein gradient flows} \label{sec:2.1}

We briefly recall the Wasserstein gradient-flow framework and some related results that motivate the present work. The quadratic Wasserstein space $(\mathcal P_2(\mathbb R^d),W_2)$ provides a natural metric setting for probability measures with finite second moment. In this space, absolutely continuous curves can be represented by velocity fields through the continuity equation
\[
  \partial_t\rho_t+\nabla\cdot(\rho_t v_t)=0,
\]
and the metric derivative $|\rho'|(t)$ is identified with the minimal $L^2(\rho_t)$-norm of such velocity fields. This dynamic interpretation is one of the basic links between optimal transport and PDEs (see for instance \cite{AGS2005,santambrogio2017euclidean,Vil}).

The first fundamental example is the variational formulation of the Fokker--Planck equation by the JKO scheme \cite{JKO}. For the free energy
\[
\mathcal F(\rho) := \int_{\bbr^d}\rho(x)\ln\rho(x)\,dx
+ \int_{\bbr^d} V(x)\,d\rho(x),
\]
with the usual convention that the entropy is $+\infty$ if $\rho$ is not absolutely continuous, the Wasserstein gradient flow formally yields
\[
\partial_t\rho = \nabla\cdot(\rho\nabla V)+\Delta\rho.
\]
Under suitable assumptions on $V$, the JKO scheme constructs a solution to this flow by the implicit variational iteration
\[
\rho_\tau^k \in \argmin_{\rho\in \mathcal P_2(\bbr^d)}
\left\{
\frac{1}{2\tau}W_2^2(\rho,\rho_\tau^{k-1})+\mathcal F(\rho)
\right\},
\]
where $\tau>0$ is a fixed time step and $\rho_\tau^0:=\rho_0$. It has become one of the standard tools for constructing solutions to dissipative PDEs in probability spaces.

Another important development is Otto's formal Riemannian calculus on $\mathcal P_2(\mathbb R^d)$ \cite{otto2001geometry}. In this viewpoint, diffusion equations (for instance, the porous medium equation) can be understood as gradient flows of suitable energy functionals. This geometric interpretation is particularly useful for deriving dissipation identities and asymptotic convergence estimates.

The rigorous metric theory was developed by Ambrosio, Gigli and Savar\'e \cite{AGS2005}. In this framework, gradient flows can be formulated without relying on a smooth differentiable structure. The main objects are the metric derivative, the local slope $|\partial\mathcal F|$, strong upper gradients, the energy dissipation inequality, etc (see Section \ref{sec:2.3}). This formulation is especially useful for measure-valued solutions where classical pointwise differentiability may not be available.

Wasserstein gradient flows also arise naturally for nonlocal interaction energies. A classical
pairwise interaction energy has the form
\[
\mathcal W(\rho) := \frac12\int_{\bbr^d\times \bbr^d}W(x-y) \d\rho(x) \d\rho(y),
\]
and the associated formal velocity field is
\[
v[\rho](x)=-\nabla W*\rho(x).
\]
For instance, the authors in \cite{carrillo2011global} developed a global-in-time theory
of weak measure solutions for nonlocal interaction equations and related it to Wasserstein
gradient flows. In such pairwise models, the velocity field is determined by a two-body
kernel through convolution with the measure.

The interaction energy considered in the present paper is of a different type. It is generated
by the squared volume of $n$-simplices and is therefore an $(n+1)$-body higher-order
interaction. Hence, at the formal level, one may expect a genuinely higher-order nonlocal
continuity equation. The point of the subsequent analysis is that this expectation is too
pessimistic: the covariance representation of $\mathcal V_n$ allows us to compute the first
variation explicitly, identify an affine velocity field, and reduce the flow to a finite-dimensional
covariance eigenvalue system.

\subsection{The functional $\mathcal V_n$ on $\mathcal P_2(\mathbb R^d)$}

In this subsection, we recall the extended generalized variance functional introduced in \cite{pronzato2017extended}; see also \cite{WZ} for a recent exterior-algebra viewpoint. For measures in $\mathcal P_2(\bbr^d)$, this functional admits a finite-dimensional representation in terms of the covariance matrix, and this representation will be the starting point of our later analysis.

Moreover, \cite{pronzato2017extended,WZ} give an explicit formula for
$\mathcal V_n[\rho]$ in terms of the eigenvalues of $\Sigma(\rho)$.

\begin{proposition}[Moment representation of $\mathcal V_n$ {\cite{pronzato2017extended,WZ}}]
	\label{prop:moment-representation}
	Let $1\le n\le d$ and $\rho\in\mathcal P_2(\mathbb R^d)$. Then
	\[
	\mathcal V_n[\rho]
	=
	\frac{n+1}{n!}\,
	e_n\bigl(\lambda_1(\rho),\dots,\lambda_d(\rho)\bigr),
	\]
	where $\lambda_1(\rho),\dots,\lambda_d(\rho)$ are the eigenvalues of $\Sigma(\rho)$, and
	\[
	e_n(c_1,\dots,c_d)
	:=
	\sum_{1\le i_1<\cdots<i_n\le d} c_{i_1}\cdots c_{i_n}.
	\]
	In particular, $\mathcal V_n[\rho]<\infty$.
\end{proposition}

The previous representation immediately yields a useful quantitative bound in terms of the second moment
\[
M_2(\rho):=\int_{\mathbb R^d}|x|^2\,\d\rho(x).
\]

\begin{corollary}
	\label{cor:Vn-upper-bound}
	For every $\rho\in\mathcal P_2(\mathbb R^d)$,
	\[
	\mathcal V_n[\rho]
	\le
	\frac{n+1}{(n!)^2}\bigl(\operatorname{tr}\Sigma(\rho)\bigr)^n
	\le
	\frac{n+1}{(n!)^2}M_2(\rho)^n.
	\]
\end{corollary}

\begin{proof}
	Since the eigenvalues of $\Sigma(\rho)$ are nonnegative, we have
	\[
	n!\,e_n(\lambda_1,\dots,\lambda_d)
	\le
	(\lambda_1+\cdots+\lambda_d)^n
	=
	\bigl(\operatorname{tr}\Sigma(\rho)\bigr)^n.
	\]
	Combining this with Proposition~\ref{prop:moment-representation}, we obtain
	\[
	\mathcal V_n[\rho]
	\le
	\frac{n+1}{(n!)^2}\bigl(\operatorname{tr}\Sigma(\rho)\bigr)^n.
	\]
	The second inequality follows from
	\[
	\operatorname{tr}\Sigma(\rho)
	=
	\int_{\mathbb R^d}|x-m(\rho)|^2\,\d\rho(x)
	\le
	\int_{\mathbb R^d}|x|^2\,\d\rho(x)
	=
	M_2(\rho).
	\]
\end{proof}

The covariance representation also shows that $\mathcal V_n$ is continuous with respect to the quadratic Wasserstein distance.

\begin{lemma}
	\label{lem:Vn-W2-continuity}
	The functional $\mathcal V_n$ is continuous on $\mathcal P_2(\mathbb R^d)$ with respect to $W_2$.
\end{lemma}

\begin{proof} 
By Proposition~\ref{prop:moment-representation}, \[ \mathcal V_n[\rho] = \frac{n+1}{n!}\,e_n(\Sigma(\rho)). \] 
Let $\rho_k\to\rho$ in $W_2$. By \cite[Theorem~6.9]{Vil}, for any continuous function $\psi$ satisfying 
\[ |\psi(x)|\le C(1+|x|^2), \] 
we have 
\[ \int_{\mathbb R^d}\psi(x)\,\d\rho_k(x)\to \int_{\mathbb R^d}\psi(x)\,\d\rho(x). \] 
Applying this to $\psi(x)=x_i$ and $\psi(x)=x_ix_j$, we obtain
\[ m(\rho_k)\to m(\rho) \qquad\text{and}\qquad \int_{\mathbb R^d}xx^\top\,\d\rho_k(x)\to \int_{\mathbb R^d}xx^\top\,\d\rho(x) \] 
entrywise. Since 
\[ \Sigma(\rho)=\int_{\mathbb R^d}xx^\top\,\d\rho(x)-m(\rho)m(\rho)^\top, \] 
it follows that \[ \Sigma(\rho_k)\to\Sigma(\rho) \] entrywise. To see the continuity of $e_n$, observe that 
\[ e_n(\lambda_1(A),\dots,\lambda_d(A)) \] 
is the coefficient of $t^n$ in $\det(I+tA)$. Hence the map 
\[ A\mapsto e_n(\lambda_1(A),\dots,\lambda_d(A)) \] 
is a polynomial in the entries of $A$, and is therefore continuous with respect to entrywise convergence. Consequently, \[ e_n(\Sigma(\rho_k))\to e_n(\Sigma(\rho)). \] 
Therefore, \[ \mathcal V_n[\rho_k]\to \mathcal V_n[\rho]. \] 
\end{proof}

In Section~\ref{sec:3}, we will compute the first variation of $\mathcal V_n$ by differentiating its covariance representation along perturbations of the underlying measure. For this purpose, we also recall the differential formula for the elementary symmetric polynomial.

\begin{proposition}[Differential of $e_n$ {\cite[equation (2.3)]{gursky2003fully}}]
	\label{prop:differential-en}
	Let $1\le n\le d$. For a symmetric matrix $\Sigma\in\mathbb R^{d\times d}$, let $e_n(\Sigma)$ denote the $n$-th elementary symmetric polynomial of the eigenvalues of $\Sigma$. Then, for every symmetric matrix $A\in\mathbb R^{d\times d}$,
	\[
	\d e_n(\Sigma)[A]
	=
	\operatorname{tr}\bigl(T_{n-1}(\Sigma)A\bigr),
	\]
	where $T_{n-1}$ is the $(n-1)$-th Newton transformation  defined by
	\[
	T_{n-1}(\Sigma)
	:=
	\sum_{k=0}^{n-1}(-1)^k e_{n-1-k}(\Sigma)\Sigma^k.
	\]
	Here we use the convention $e_0(\Sigma)=1$ and $\Sigma^0=I$.
\end{proposition}

Since $T_{n-1}(\Sigma)$ is a polynomial in $\Sigma$, it is symmetric whenever $\Sigma$ is symmetric and commutes with $\Sigma$. This simple observation will be used repeatedly when deriving the closed evolution equation for the covariance matrix.

\subsection{Wasserstein framework and solution notions} \label{sec:2.3}

In this subsection, we collect the basic notions from the quadratic Wasserstein framework that will be used later. 

Throughout this subsection, $(\mathcal P_2(\mathbb R^d),W_2)$ denotes the quadratic Wasserstein space over $\mathbb R^d$. The solution framework used in this paper involves two related notions. On the PDE side, we use weak solutions to continuity equations, which are natural for the explicit construction by transition maps. On the variational side, we use maximal slope solutions in the sense of metric gradient flows. We will later
show that, for the functional $\mathcal V_n$, the two formulations are connected through the explicit identification of the Wasserstein gradient field.

\begin{definition}[Absolutely continuous curves {\cite[Definition~1.1.1]{AGS2005}}] 
Let $(X,d)$ be a complete metric space and let $p\in[1,\infty]$. We denote by $AC^p(a,b;X)$ the set of all curves $\gamma:(a,b)\to X$ for which there exists a function $g\in L^p(a,b)$ such that 
\[ d(\gamma(s),\gamma(t)) \le \int_s^t g(r)\,\d r, \qquad a<s<t<b. \] 
\end{definition} 

For absolutely continuous curves, one can define an intrinsic notion of speed
depending only on the metric structure. This is given by the metric derivative,
which exists for almost every time along such curves.

\begin{proposition}[Metric derivative {\cite[Theorem~1.1.2]{AGS2005}}]
	Let $(X,d)$ be a complete metric space and let $p\in[1,\infty]$. If
	$\gamma\in AC^p(a,b;X)$, then the limit
	\[
	|\gamma'|(t):=\lim_{h\to0}\frac{d(\gamma(t+h),\gamma(t))}{|h|}
	\]
	exists for a.e.\ $t\in(a,b)$, and the function $|\gamma'|$ belongs to
	$L^p(a,b)$. Moreover,
	\[
	d(\gamma(s),\gamma(t))
	\le
	\int_s^t |\gamma'|(r)\,\d r,
	\qquad a<s<t<b.
	\]
	In addition, if $g\in L^p(a,b)$ satisfies
	\[
	d(\gamma(s),\gamma(t))
	\le
	\int_s^t g(r)\,\d r,
	\qquad a<s<t<b,
	\]
	then $|\gamma'|(t)\le g(t)$ for a.e.\ $t\in(a,b)$.
\end{proposition}

In what follows, we will mainly work with curves in
$AC^2(0,T;\mathcal P_2(\mathbb R^d))$. 
In the metric formulation of gradient flows, the metric derivative plays the
role of the speed of a curve. The corresponding notion for a functional is the
slope, which measures the maximal instantaneous decrease of the functional per
unit distance.

\begin{definition}[Slope {\cite[Definition~1.2.4]{AGS2005}}]
	Let $(X,d)$ be a complete metric space and let $\phi:X\to(-\infty,\infty]$. The slope of $\phi$ at $v\in X$ is defined by
	\[
	|\partial\phi|(v)
	:=
	\limsup_{w\to v}\frac{(\phi(v)-\phi(w))_+}{d(v,w)}.
	\]
\end{definition}

The slope is the basic object used to formulate gradient-flow solutions in a
metric space. Before introducing the PDE formulation, we also recall the
notion of narrow continuity for curves of probability measures.

\begin{definition}[Narrow continuity]
	A curve $(\rho_t)_{t\in[0,T]}\subset\mathcal P(\mathbb R^d)$ is said to be narrowly continuous if
	\[
	t\longmapsto \int_{\mathbb R^d}\varphi(x)\,\d\rho_t(x)
	\]
	is continuous on $[0,T]$ for every $\varphi\in C_b(\mathbb R^d)$.
\end{definition}

A basic fact in Wasserstein geometry is that absolutely continuous curves can be
represented by velocity fields through the continuity equation. The following
result is a special case, with $p=2$, of the general continuity-equation
characterization of absolutely continuous curves in Wasserstein spaces. Since
the present paper is concerned with the quadratic Wasserstein space, we record
only the $W_2$ version.

\begin{proposition}[Continuity equation representation {\cite[Theorem~8.3.1]{AGS2005}}] \label{prop:continuity-equation} 
Let $(\rho_t)_{t\in[0,T]}\subset\mathcal P_2(\mathbb R^d)$. \begin{enumerate} 
\item If $\rho\in AC^2(0,T;(\mathcal P_2(\mathbb R^d),W_2))$, then there exists a Borel vector field 
\[ v:(0,T)\times\mathbb R^d\to\mathbb R^d, \qquad v_t(x):=v(t,x), \] 
such that 
\[ \int_0^T \|v_t\|_{L^2(\rho_t;\mathbb R^d)}^2\,\d t<\infty, \] and 
\[ \partial_t\rho_t+\nabla\cdot(\rho_t v_t)=0 \] 
holds in the sense of distributions on $(0,T)\times\mathbb R^d$. Moreover, 
\[ \|v_t\|_{L^2(\rho_t;\mathbb R^d)} \le |\rho'|(t) \qquad\text{for a.e.\ }t\in(0,T). \] 
\item Conversely, if $(\rho_t)_{t\in[0,T]}$ is narrowly continuous and there exists a Borel vector field
\[ v:(0,T)\times\mathbb R^d\to\mathbb R^d, \qquad v_t(x):=v(t,x), \] 
such that 
\[ \partial_t\rho_t+\nabla\cdot(\rho_t v_t)=0 \] 
in the sense of distributions on $(0,T)\times\mathbb R^d$ and 
\[ \int_0^T \|v_t\|_{L^2(\rho_t;\mathbb R^d)}^2\,\d t<\infty, \] then $\rho\in AC^2(0,T;(\mathcal P_2(\mathbb R^d),W_2))$ and 
\[ |\rho'|(t)\le \|v_t\|_{L^2(\rho_t;\mathbb R^d)} \qquad\text{for a.e.\ }t\in(0,T). \] 
\end{enumerate} 
\end{proposition}

We shall also use the following first-order approximation property of
absolutely continuous curves in Wasserstein space. It says that, at almost
every time, the infinitesimal displacement of the curve is represented to
first order by pushing the measure forward along its tangent velocity field.

\begin{proposition}[First-order approximation by the tangent velocity field {\cite[Proposition~8.4.6]{AGS2005}}]
\label{prop:first-order-approximation}

Let $\mu\in AC^2(0,T;(\mathcal P_2(\mathbb R^d),W_2))$, and let $w_t$ be the minimal velocity field associated with $\mu$, whose existence is guaranteed by Proposition~\ref{prop:continuity-equation}(1). Thus,
\[
\partial_t\mu_t+\nabla\cdot(\mu_t w_t)=0
\]
holds in the sense of distributions on $(0,T)\times\mathbb R^d$, and
\[
\|w_t\|_{L^2(\mu_t)}=|\mu'|(t)
\qquad\text{for a.e. }t\in(0,T).
\]
Then, for a.e. $t\in(0,T)$,
\[
\lim_{h\to0}
\frac{
W_2\bigl(\mu_{t+h},(\mathrm{Id}+h w_t)_\#\mu_t\bigr)
}{|h|}
=0.
\]
\end{proposition}

The previous results show that $AC^2$ curves in the quadratic Wasserstein space admit square-integrable velocity fields. The weak formulation of the continuity equation itself, however, only requires narrow continuity of the curve, provided that the terms in the weak formulation below are well defined. We now state this formulation for a prescribed velocity field and an initial datum; the initial condition is incorporated by allowing test functions that may be nonzero at $t=0$.

\begin{definition}[Weak solution to the continuity equation] \label{def:weak-solution} 
Let $\rho_0\in\mathcal P_2(\mathbb R^d)$ and let 
\[ v:(0,T)\times\mathbb R^d\to\mathbb R^d, \qquad v_t(x):=v(t,x), \] 
be a Borel vector field such that 
\[ \int_0^{T'}\int_K |v_t(x)|\,\d\rho_t(x)\,\d t<\infty \] 
for every $T'\in(0,T)$ and every compact set $K\subset\mathbb R^d$. A narrowly continuous curve $(\rho_t)_{t\in[0,T]}\subset\mathcal P_2(\mathbb R^d)$ is called a weak solution to 
\[ \partial_t\rho_t+\nabla\cdot(\rho_t v_t)=0, \qquad \rho|_{t=0}=\rho_0, \] 
if 
\[ \int_0^T\int_{\mathbb R^d} \bigl(\partial_t\varphi(t,x)+v_t(x)\cdot\nabla\varphi(t,x)\bigr)\, \d\rho_t(x)\,\d t + \int_{\mathbb R^d}\varphi(0,x)\,\d\rho_0(x)=0 \] 
for every $\varphi\in C_c^\infty([0,T)\times\mathbb R^d)$. \end{definition}

We now turn to the variational formulation of gradient flows in metric spaces. The key object in this formulation is a strong upper gradient, which controls the variation of a functional along absolutely continuous curves.

\begin{definition}[Strong upper gradient {\cite[Definition~1.2.1]{AGS2005}}]
	Let $(X,d)$ be a complete metric space and let $\phi:X\to(-\infty,\infty]$. A function
	\[
	g:X\to[0,\infty]
	\]
	is called a strong upper gradient for $\phi$ if for every $\gamma\in AC(a,b;X)$,
	\[
	|\phi(\gamma(t))-\phi(\gamma(s))|
	\le
	\int_s^t g(\gamma(r))\,|\gamma'|(r)\,\d r,
	\qquad
	a<s\le t<b.
	\]
\end{definition}

Once a strong upper gradient is available, gradient-flow solutions can be formulated by an energy dissipation inequality involving the metric derivative of the curve and the upper gradient of the functional. This leads to the following notion of maximal slope solution.

\begin{definition}[Maximal slope solution {\cite[Definition~1.3.2]{AGS2005}}]
	\label{def:maximal-slope}
	Let
	\[
	F:\mathcal P_2(\mathbb R^d)\to(-\infty,\infty]
	\]
	be a functional with $F^{-1}(\mathbb{R})\neq \varnothing$, and let $g$ be a strong upper gradient for $F$. A curve
	\[
	\rho:[0,T]\to\mathcal P_2(\mathbb R^d)
	\]
	is called a maximal slope solution for $F$ with respect to $g$ if
	\[
	\rho\in AC^2(0,T;(\mathcal P_2(\mathbb R^d),W_2)),
	\qquad
	F(\rho_t)<\infty \quad \text{for all } t\in[0,T],
	\]
	and there exists a nonincreasing function $\varphi:[0,T]\to\mathbb R$ such that
	\[
	\varphi(t)=F(\rho_t)\quad\text{for a.e.\ }t\in(0,T)
	\]
	and
	\[
	\varphi'(t)\le -\frac12|\rho'|^2(t)-\frac12 g^2(\rho_t)
	\qquad\text{for a.e.\ }t\in(0,T).
	\]
\end{definition}

By the monotonicity of $\varphi$, the differential inequality in Definition~\ref{def:maximal-slope} yields the integrated energy dissipation inequality 
\[ \varphi(t) + \frac12\int_s^t |\rho'|^2(r)\,\d r + \frac12\int_s^t g^2(\rho_r)\,\d r \le \varphi(s), \qquad 0\le s\le t\le T. \] 
 In particular, if $t\mapsto F(\rho_t)$ is continuous, then the representative $\varphi$ can be identified with $F(\rho_t)$ for all $t$, and hence 
\[ F(\rho_t) + \frac12\int_s^t |\rho'|^2(r)\,\d r + \frac12\int_s^t g^2(\rho_r)\,\d r \le F(\rho_s), \qquad 0\le s\le t\le T. \]

%In the applications below, the relevant upper gradient will be $g=|\partial F|$.

In the metric-space framework of \cite{AGS2005}, maximal slope solutions play the role of gradient flows. In a smooth Hilbert setting, if $F$ is differentiable, then the slope coincides with the norm of the classical gradient, 
\[ |\partial F|(u)=\|\nabla F(u)\|. \] 
Moreover, the map $u\mapsto\|\nabla F(u)\|$ is a strong upper gradient, by the usual chain rule and the Cauchy--Schwarz inequality. Therefore the maximal slope condition with $g=|\partial F|$ reduces to the usual energy dissipation relation associated with the differential equation 
\[ u'(t)=-\nabla F(u(t)). \] 
Similarly, whenever the  slope is a strong upper gradient, the notion of maximal slope solution may be regarded as the natural variational extension of the classical gradient flow to the Wasserstein setting, where the metric derivative $|\rho'|(t)$ and the slope $|\partial F|(\rho_t)$ replace the velocity norm and the gradient norm, respectively.

Accordingly, in the present paper we regard a maximal slope solution for $F$ as the variational formulation of the Wasserstein gradient flow generated by $F$. However, this formulation is still abstract: it characterizes the evolution through energy dissipation, but it does not by itself identify the concrete velocity field appearing in the continuity equation. For the specific functional $F=\mathcal V_n$, the main task of the subsequent analysis is to make this variational description explicit.

More precisely, in the next section we will compute the first variation of $\mathcal V_n$ and show that the associated vector field can be written explicitly in terms of the mean $m(\rho)$ and the covariance matrix $\Sigma(\rho)$. This will yield a closed finite-dimensional evolution equation for $\Sigma(\rho_t)$. Solving this finite-dimensional system, we will construct a candidate curve by pushing forward the initial datum through the corresponding transition map. We will then verify that the resulting curve is both a weak solution to the continuity equation and a maximal slope solution for $\mathcal V_n$. Finally, after establishing the relevant chain rule and slope identity for $\mathcal V_n$, we will show that any maximal slope solution must be driven by the same explicit velocity field. In this sense, for the flow generated by $\mathcal V_n$, the abstract variational formulation and the concrete continuity-equation formulation coincide.

The next proposition isolates the abstract mechanism behind this final identification. It shows that once a chain rule and a slope identity are available along a maximal slope solution, the velocity field in the continuity equation is uniquely determined.

\begin{proposition}[Abstract velocity identification] \label{prop:abstract-identification} 
Let 
\[ F:\mathcal P_2(\mathbb R^d)\to(-\infty,\infty] \] 
be a functional with $F^{-1}(\mathbb R)\neq\varnothing$. Assume that the local slope $|\partial F|$ is a strong upper gradient for $F$, and let $(\rho_t)_{t\in[0,T]}$ be a maximal slope solution for $F$ with respect to $|\partial F|$. Let $\varphi$ be the nonincreasing representative appearing in Definition~\ref{def:maximal-slope}. Let $w_t$ be the minimal velocity field associated with $(\rho_t)$, so that 
\[ \partial_t\rho_t+\nabla\cdot(\rho_t w_t)=0 \qquad\text{in }\mathcal D'((0,T)\times\mathbb R^d), \] 
and 
\[ \|w_t\|_{L^2(\rho_t)}=|\rho'|(t) \qquad\text{for a.e. }t\in(0,T). \] 
Assume moreover that there exists a vector field 
\[ \xi_{\rho_t}\in L^2(\rho_t;\mathbb R^d) \] 
such that, along the curve $(\rho_t)$, one has the chain rule 
\[ \varphi'(t) = \int_{\mathbb R^d} \langle \xi_{\rho_t}(x),w_t(x)\rangle\,\d\rho_t(x) \qquad\text{for a.e. }t\in(0,T), \]
and the slope identity 
\[ \|\xi_{\rho_t}\|_{L^2(\rho_t)}=|\partial F|(\rho_t) \qquad\text{for a.e. }t\in(0,T). \] 
Then \[ w_t=-\xi_{\rho_t} \qquad\text{for a.e. }t\in(0,T). \] \end{proposition}

\begin{proof}
	Since $(\rho_t)$ is a maximal slope solution with respect to
	$g=|\partial F|$, $\varphi$ satisfies
	\[
	-\varphi'(t)
	\ge
	\frac12|\rho'|^2(t)+\frac12|\partial F|^2(\rho_t)
	\qquad\text{for a.e. }t\in(0,T).
	\]
	On the other hand, by the chain rule and the Cauchy--Schwarz inequality,
	\[
	-\varphi'(t)
	=
	-\int_{\mathbb R^d}
	\langle \xi_{\rho_t},w_t\rangle\,d\rho_t
	\le
	\|\xi_{\rho_t}\|_{L^2(\rho_t)}\|w_t\|_{L^2(\rho_t)}.
	\]
	Using Young's inequality and the slope identity, we obtain
	\[
	-\varphi'(t)
	\le
	\frac12\|\xi_{\rho_t}\|_{L^2(\rho_t)}^2
	+
	\frac12\|w_t\|_{L^2(\rho_t)}^2
	=
	\frac12|\partial F|^2(\rho_t)
	+
	\frac12|\rho'|^2(t).
	\]
	Thus equality holds throughout. In particular, equality holds in the
	Cauchy--Schwarz and Young inequalities, and hence
	\[
	w_t=-\xi_{\rho_t}
	\qquad\text{for a.e. }t\in(0,T).
	\]
\end{proof}

\section{Finite-dimensional reduction and construction of the flow}\label{sec:3}
\setcounter{equation}{0}

In this section, we exploit the covariance representation of $\mathcal V_n$ to reduce the Wasserstein gradient flow to a finite-dimensional dynamical system. More precisely, we first compute the first variation of $\mathcal V_n$ and identify the associated velocity field explicitly in terms of the mean and the covariance matrix. This yields a closed evolution equation for the covariance matrix, from which we construct a candidate curve by pushforward of the initial datum through a time-dependent linear transition map. We then show that this curve solves the continuity equation, satisfies the maximal slope formulation for $\mathcal V_n$, and is uniquely determined by the initial datum.

\subsection{First variation, covariance dynamics, and pushforward construction}\label{sec:3.1}

We begin by computing the first variation of $\mathcal V_n$. Owing to the moment representation established in Section~2.2, this computation can be carried out at the level of the covariance matrix. As a consequence, the Wasserstein gradient field associated with $\mathcal V_n$ turns out to be affine in the space variable, with coefficients depending only on the mean $m(\rho)$ and the covariance matrix $\Sigma(\rho)$. This structural simplification is the key point of the analysis, since it reduces the infinite-dimensional gradient flow problem to a closed finite-dimensional system for $\Sigma(\rho_t)$.

\begin{proposition}[First variation of $\mathcal V_n$]
	\label{prop:first-variation-Vn}
	Let $\rho\in\mathcal P_2(\mathbb R^d)$, and write
	\[
	m:=m(\rho),
	\qquad
	\Sigma:=\Sigma(\rho),
	\qquad
	T:=T_{n-1}(\Sigma).
	\]
	For $\xi\in C_c^1(\mathbb R^d;\mathbb R^d)$, define
	\[
	T_t(x):=x+t\xi(x),
	\qquad
	\rho_t:=(T_t)_\#\rho .
	\]
	Then the map $t\mapsto \mathcal V_n[\rho_t]$ is differentiable at $t=0$, and
	\[
	\frac{\d}{\d t}\mathcal V_n[\rho_t]\Big|_{t=0}
	=
	2\left(\frac{n+1}{n!}\right)
	\int_{\mathbb R^d}
	\big\langle T(x-m),\xi(x)\big\rangle\,\d\rho(x).
	\]
	Equivalently, the first variation of $\mathcal V_n$ is given by
	\[
	\frac{\delta\mathcal V_n}{\delta\rho}(x)
	=
	\frac{n+1}{n!}\,(x-m(\rho))^\top T_{n-1}(\Sigma(\rho))(x-m(\rho))
	\]
	up to an additive constant.
\end{proposition}

\begin{proof}
	By Proposition~\ref{prop:moment-representation},
	\[
	\mathcal V_n[\rho_t]
	=
	\frac{n+1}{n!}\,e_n(\Sigma(\rho_t)).
	\]
	Thus it suffices to compute the derivative of $\Sigma(\rho_t)$ at $t=0$.\\
	
	Set
	\[
	a:=\int_{\mathbb R^d}\xi(x)\,\d\rho(x).
	\]
	Since $\rho_t=(T_t)_\#\rho$, we have
	\[
	m(\rho_t)
	=
	\int_{\mathbb R^d}T_t(x)\,\d\rho(x)
	=
	\int_{\mathbb R^d}(x+t\xi(x))\,\d\rho(x)
	=
	m+t a.
	\]
	Hence
	\[
	T_t(x)-m(\rho_t)
	=
	(x-m)+t(\xi(x)-a).
	\]
	Therefore,
	\[
	\Sigma(\rho_t)
	=
	\int_{\mathbb R^d}
	(T_t(x)-m(\rho_t))(T_t(x)-m(\rho_t))^\top\,\d\rho(x)
	\]
	can be expanded as
	\[
	\Sigma(\rho_t)
	=
	\Sigma+tB+t^2C,
	\]
	where
	\[
	B
	:=
	\int_{\mathbb R^d}
	\Big((x-m)(\xi(x)-a)^\top+(\xi(x)-a)(x-m)^\top\Big)\,\d\rho(x),
	\]
	and
	\[
	C
	:=
	\int_{\mathbb R^d}
	(\xi(x)-a)(\xi(x)-a)^\top\,\d\rho(x).
	\]
	In particular,
	\[
	\frac{\d}{\d t}\Sigma(\rho_t)\Big|_{t=0}=B.
	\]
	Since $\int_{\mathbb R^d}(x-m)\,\d\rho(x)=0$, the terms involving $a$ vanish, and thus
	\[
	B
	=
	\int_{\mathbb R^d}
	\Big((x-m)\xi(x)^\top+\xi(x)(x-m)^\top\Big)\,\d\rho(x).
	\]
	Now Proposition~\ref{prop:differential-en} yields
	\[
	\frac{\d}{\d t}e_n(\Sigma(\rho_t))\Big|_{t=0}
	=
	\d e_n(\Sigma)[B]
	=
	\operatorname{tr}(TB).
	\]
	Hence
	\[
	\frac{\d}{\d t}\mathcal V_n[\rho_t]\Big|_{t=0}
	=
	\frac{n+1}{n!}\operatorname{tr}(TB).
	\]
	Using the symmetry of $T$, we compute
	\[
	\operatorname{tr}(TB)
	=
	\int_{\mathbb R^d}
	\operatorname{tr}\Big(T(x-m)\xi(x)^\top+T\xi(x)(x-m)^\top\Big)\,\d\rho(x)
	=
	2\int_{\mathbb R^d}\langle T(x-m),\xi(x)\rangle\,\d\rho(x).
	\]
	Therefore,
	\[
	\frac{\d}{\d t}\mathcal V_n[\rho_t]\Big|_{t=0}
	=
	2\left(\frac{n+1}{n!}\right)
	\int_{\mathbb R^d}
	\big\langle T(x-m),\xi(x)\big\rangle\,\d\rho(x).
	\]
	Finally, since
	\[
	\nabla_x\!\left((x-m)^\top T(x-m)\right)=2T(x-m)
	\]
	for symmetric $T$, the above identity shows that
	\[
	\frac{\delta\mathcal V_n}{\delta\rho}(x)
	=
	\frac{n+1}{n!}\,(x-m(\rho))^\top T_{n-1}(\Sigma(\rho))(x-m(\rho))
	\]
	up to an additive constant.
\end{proof}

The previous proposition identifies the first variation of $\mathcal V_n$ explicitly. In particular, taking the spatial gradient yields a velocity field which is affine in the space variable and depends only on the first two moments of the measure. This is precisely the structural feature that allows us to reduce the Wasserstein gradient flow to a finite-dimensional dynamical system.

\begin{corollary}[Explicit gradient field]
	\label{cor:explicit-gradient-field}
	For $\rho\in\mathcal P_2(\mathbb R^d)$,
	\[
	\nabla\!\left(\frac{\delta\mathcal V_n}{\delta\rho}\right)(x)
	=
	2\left(\frac{n+1}{n!}\right)\,T_{n-1}(\Sigma(\rho))(x-m(\rho)).
	\]
	Accordingly, the associated velocity field is
	\[
	v[\rho](x)
	:=
	-\nabla\!\left(\frac{\delta\mathcal V_n}{\delta\rho}\right)(x)
	=
	-2\left(\frac{n+1}{n!}\right)\,T_{n-1}(\Sigma(\rho))(x-m(\rho)).
	\]
\end{corollary}

\begin{proof}
	This follows immediately from Proposition~\ref{prop:first-variation-Vn} and the symmetry of
	$T_{n-1}(\Sigma(\rho))$.
\end{proof}

The explicit form of the velocity field allows us to derive closed evolution equations for the first two moments of the flow. More precisely, along the continuity equation driven by this field, the mean is conserved, while the covariance matrix satisfies an autonomous matrix ODE involving the Newton transformation $T_{n-1}$. Passing to the eigenvalue variables, this matrix equation reduces to a closed system of scalar ODEs. Solving the reduced system, we then construct a time-dependent transition map and define a candidate solution curve by pushing forward the initial datum.\\

From now on, let $(\rho_t)_{t\in[0,T]}\subset \mathcal P_2(\mathbb R^d)$ be a weak solution to \[ \partial_t\rho_t+\nabla\cdot(\rho_t v[\rho_t])=0 \] with the velocity field $v[\rho]$ given by Corollary~\ref{cor:explicit-gradient-field}. We write \[ m_t:=m(\rho_t), \qquad \Sigma_t:=\Sigma(\rho_t). \] Since $v[\rho](x)$ is affine in $x$, for each fixed $t$ we have \[ |v[\rho_t](x)|\le C_t(1+|x|), \] and therefore \[ \int_{\mathbb R^d}|v[\rho_t](x)|\,\d\rho_t(x) + \int_{\mathbb R^d}|x|\,|v[\rho_t](x)|\,\d\rho_t(x) <\infty. \] This allows us to justify the moment identities below by means of cutoff approximations.

\begin{proposition}[Conservation of the mean] \label{prop:mean-conservation} Let $(\rho_t)_{t\in[0,T]}\subset \mathcal P_2(\mathbb R^d)$ be a weak solution to \[ \partial_t\rho_t+\nabla\cdot(\rho_t v[\rho_t])=0, \] where \[ v[\rho](x) = -2\left(\frac{n+1}{n!}\right)\,T_{n-1}(\Sigma(\rho))(x-m(\rho)). \] Then \[ \frac{\d}{\d t}m_t=0 \qquad\text{for a.e. }t\in(0,T). \] In particular, \[ m_t=m_0 \qquad\text{for all }t\in[0,T]. \] \end{proposition} \begin{proof} Fix $j\in\{1,\dots,d\}$. Let $\eta\in C_c^\infty(\mathbb R^d)$ satisfy \[ 0\le \eta\le 1, \qquad \eta(x)=1 \ \text{for } |x|\le 1, \qquad \eta(x)=0 \ \text{for } |x|\ge 2. \] For $R>1$, define \[ \eta_R(x):=\eta\!\left(\frac{x}{R}\right), \qquad \phi_{R,j}(x):=x_j\eta_R(x). \] Then $\phi_{R,j}\in C_c^\infty(\mathbb R^d)$, and \[ \nabla\phi_{R,j}(x) = e_j\,\eta_R(x)+x_j\nabla\eta_R(x), \] where $e_j$ is the $j$-th standard basis vector. Since \[ |\nabla\eta_R(x)|\le \frac{C}{R}\mathbf 1_{\{R\le |x|\le 2R\}}, \] we obtain \[ |\nabla\phi_{R,j}(x)| \le 1 + \frac{C|x_j|}{R}\mathbf 1_{\{R\le |x|\le 2R\}} \le C \] for a constant $C$ independent of $R$. Now let $\zeta\in C_c^\infty(0,T)$. Since $\rho_t$ is a weak solution, using the test function \[ \varphi(t,x)=\zeta(t)\phi_{R,j}(x), \] we obtain \[ \int_0^T \zeta'(t)\int_{\mathbb R^d}\phi_{R,j}(x)\,\d\rho_t(x)\,\d t + \int_0^T \zeta(t)\int_{\mathbb R^d}\nabla\phi_{R,j}(x)\cdot v[\rho_t](x)\,\d\rho_t(x)\,\d t =0. \] That is, \[ \frac{\d}{\d t}\int_{\mathbb R^d}\phi_{R,j}(x)\,\d\rho_t(x) = \int_{\mathbb R^d}\nabla\phi_{R,j}(x)\cdot v[\rho_t](x)\,\d\rho_t(x) \] in the sense of distributions on $(0,T)$. We now pass to the limit $R\to\infty$. Since $\phi_{R,j}(x)\to x_j$ pointwise and \[ |\phi_{R,j}(x)|\le |x_j|\le |x|, \] the dominated convergence theorem gives \[ \int_{\mathbb R^d}\phi_{R,j}(x)\,\d\rho_t(x)\to \int_{\mathbb R^d}x_j\,\d\rho_t(x) \] for every $t\in[0,T]$, because $\rho_t\in \mathcal P_2(\mathbb R^d)\subset \mathcal P_1(\mathbb R^d)$. Next, since \[ \nabla\phi_{R,j}(x) = e_j\,\eta_R(x)+x_j\nabla\eta_R(x), \] we have \[ \nabla\phi_{R,j}(x)\to e_j \qquad\text{for every }x\in\mathbb R^d, \] and moreover \[ |\nabla\phi_{R,j}(x)\cdot v[\rho_t](x)| \le C|v[\rho_t](x)|. \] Because $v[\rho_t]\in L^1(\rho_t)$, another application of dominated convergence yields \[ \int_{\mathbb R^d}\nabla\phi_{R,j}(x)\cdot v[\rho_t](x)\,\d\rho_t(x) \to \int_{\mathbb R^d} e_j\cdot v[\rho_t](x)\,\d\rho_t(x). \] Therefore, \[ \frac{\d}{\d t}\int_{\mathbb R^d}x_j\,\d\rho_t(x) = \int_{\mathbb R^d} e_j\cdot v[\rho_t](x)\,\d\rho_t(x) \qquad\text{in }\mathcal D'(0,T). \] Since this holds for each $j=1,\dots,d$, we obtain \[ \frac{\d}{\d t}m_t = \int_{\mathbb R^d}v[\rho_t](x)\,\d\rho_t(x) \qquad\text{for a.e. }t\in(0,T). \] Finally, using the explicit form of the velocity field, \[ \int_{\mathbb R^d}v[\rho_t](x)\,\d\rho_t(x) = -2\left(\frac{n+1}{n!}\right)\,T_{n-1}(\Sigma_t) \int_{\mathbb R^d}(x-m_t)\,\d\rho_t(x) = 0. \] Hence \[ \frac{\d}{\d t}m_t=0 \qquad\text{for a.e. }t\in(0,T). \] Since $t\mapsto m_t$ is absolutely continuous, it follows that \[ m_t=m_0 \qquad\text{for all }t\in[0,T]. \] \end{proof}

The next result shows that the second-order moment also evolves according to a closed equation. Since the velocity field is affine in the space variable, the resulting equation closes at the level of the covariance matrix.

\begin{proposition}[Covariance dynamics] \label{prop:covariance-dynamics} Let $(\rho_t)_{t\in[0,T]}\subset \mathcal P_2(\mathbb R^d)$ be a weak solution to \[ \partial_t\rho_t+\nabla\cdot(\rho_t v[\rho_t])=0, \] where \[ v[\rho](x) = -2\left(\frac{n+1}{n!}\right)\,T_{n-1}(\Sigma(\rho))(x-m(\rho)). \] Then the covariance matrix satisfies \[ \frac{\d}{\d t}\Sigma_t = -4\left(\frac{n+1}{n!}\right)\,\Sigma_t\,T_{n-1}(\Sigma_t) \qquad\text{for a.e. }t\in(0,T). \] \end{proposition} \begin{proof} By Proposition~\ref{prop:mean-conservation}, we have $m_t=m_0$ for all $t\in[0,T]$. Hence \[ \Sigma_t = \int_{\mathbb R^d}(x-m_0)(x-m_0)^\top\,\d\rho_t(x). \] Fix $i,j\in\{1,\dots,d\}$. Define \[ \Psi_{ij}(x):=(x_i-m_{0,i})(x_j-m_{0,j}). \] To justify the differentiation of its moment, we again use a cutoff. Let $\eta_R$ be as in the previous proof and set \[ \Psi_{ij}^R(x):=\Psi_{ij}(x)\eta_R(x)\in C_c^\infty(\mathbb R^d). \] Using the weak formulation with the test function \[ \varphi(t,x)=\zeta(t)\Psi_{ij}^R(x), \qquad \zeta\in C_c^\infty(0,T), \] we obtain \[ \frac{\d}{\d t}\int_{\mathbb R^d}\Psi_{ij}^R(x)\,\d\rho_t(x) = \int_{\mathbb R^d}\nabla\Psi_{ij}^R(x)\cdot v[\rho_t](x)\,\d\rho_t(x) \] in the sense of distributions. We first compute the gradient of the untruncated  function: \[ \nabla\Psi_{ij}(x) = (x_j-m_{0,j})e_i+(x_i-m_{0,i})e_j. \] Hence \[ |\nabla\Psi_{ij}(x)|\le C(1+|x|). \] Moreover, \[ \nabla\Psi_{ij}^R(x) = \eta_R(x)\nabla\Psi_{ij}(x)+\Psi_{ij}(x)\nabla\eta_R(x). \] Since \[ |\Psi_{ij}(x)|\le C(1+|x|^2), \qquad |\nabla\eta_R(x)|\le \frac{C}{R}\mathbf 1_{\{R\le |x|\le 2R\}}, \] we get \[ |\nabla\Psi_{ij}^R(x)| \le C(1+|x|). \] Therefore \[ |\nabla\Psi_{ij}^R(x)\cdot v[\rho_t](x)| \le C(1+|x|)\,|v[\rho_t](x)|. \] Since $v[\rho_t](x)$ is affine in $x$, we have \[ |v[\rho_t](x)|\le C_t(1+|x|), \] and thus \[ (1+|x|)\,|v[\rho_t](x)|\le C_t(1+|x|^2). \] Because $\rho_t\in\mathcal P_2(\mathbb R^d)$, the right-hand side is integrable with respect to $\rho_t$. Now $\Psi_{ij}^R(x)\to \Psi_{ij}(x)$ pointwise and \[ |\Psi_{ij}^R(x)|\le C(1+|x|^2), \] so dominated convergence gives \[ \int_{\mathbb R^d}\Psi_{ij}^R(x)\,\d\rho_t(x) \to \int_{\mathbb R^d}\Psi_{ij}(x)\,\d\rho_t(x). \] Likewise, \[ \nabla\Psi_{ij}^R(x)\cdot v[\rho_t](x)\to \nabla\Psi_{ij}(x)\cdot v[\rho_t](x) \] pointwise, and the above bound implies \[ \int_{\mathbb R^d}\nabla\Psi_{ij}^R(x)\cdot v[\rho_t](x)\,\d\rho_t(x) \to \int_{\mathbb R^d}\nabla\Psi_{ij}(x)\cdot v[\rho_t](x)\,\d\rho_t(x). \] Hence \[ \frac{\d}{\d t}\int_{\mathbb R^d}\Psi_{ij}(x)\,\d\rho_t(x) = \int_{\mathbb R^d}\nabla\Psi_{ij}(x)\cdot v[\rho_t](x)\,\d\rho_t(x) \] for a.e. $t\in(0,T)$. Since \[ \nabla\Psi_{ij}(x) = (x_j-m_{0,j})e_i+(x_i-m_{0,i})e_j, \] this becomes \[ \frac{\d}{\d t}(\Sigma_t)_{ij} = \int_{\mathbb R^d} \Big((x_j-m_{0,j})(v_t)_i(x)+(x_i-m_{0,i})(v_t)_j(x)\Big)\,\d\rho_t(x), \] where $v_t(x):=v[\rho_t](x)$. In matrix form, \[ \frac{\d}{\d t}\Sigma_t = \int_{\mathbb R^d} \Big((x-m_0)v_t(x)^\top+v_t(x)(x-m_0)^\top\Big)\,\d\rho_t(x). \] Substituting \[ v_t(x) = -2\left(\frac{n+1}{n!}\right)\,T_{n-1}(\Sigma_t)(x-m_0), \] we obtain \[ \frac{\d}{\d t}\Sigma_t = -2\left(\frac{n+1}{n!}\right) \int_{\mathbb R^d} \Big((x-m_0)(x-m_0)^\top T_{n-1}(\Sigma_t) + T_{n-1}(\Sigma_t)(x-m_0)(x-m_0)^\top\Big)\,\d\rho_t(x). \] Since $T_{n-1}(\Sigma_t)$ is independent of $x$, this simplifies to \[ \frac{\d}{\d t}\Sigma_t = -2\left(\frac{n+1}{n!}\right) \Big( \Sigma_t T_{n-1}(\Sigma_t) + T_{n-1}(\Sigma_t)\Sigma_t \Big). \] Finally, $T_{n-1}(\Sigma_t)$ is a polynomial in $\Sigma_t$, so it commutes with $\Sigma_t$. Therefore, \[ \frac{\d}{\d t}\Sigma_t = -4\left(\frac{n+1}{n!}\right)\,\Sigma_t T_{n-1}(\Sigma_t), \] which is the desired identity. \end{proof}

We now diagonalize the covariance equation. Since the right-hand side is a polynomial in $\Sigma_t$, the eigenspaces of the initial covariance matrix are preserved, and the matrix equation reduces to a closed system for the eigenvalues.

\begin{corollary}[Eigenvalue system] \label{cor:eigenvalue-system} Let \[ \Sigma_0=P^\top D_0P, \qquad D_0=\operatorname{diag}(\lambda_1^0,\dots,\lambda_d^0), \] with $P$ orthogonal and $\lambda_i^0\ge0$. Then the solution of \[ \frac{\d}{\d t}\Sigma_t = -4\left(\frac{n+1}{n!}\right)\,\Sigma_t\,T_{n-1}(\Sigma_t) \] has the form \[ \Sigma_t=P^\top D_tP, \qquad D_t=\operatorname{diag}(\lambda_1(t),\dots,\lambda_d(t)), \] where \[ \frac{\d}{\d t}\lambda_i(t) = -4\left(\frac{n+1}{n!}\right)\,\lambda_i(t)\,e_{n-1}(\widehat\lambda_i(t)), \qquad 1\le i\le d. \] Here $e_{n-1}(\widehat\lambda_i(t))$ denotes the $(n-1)$-th elementary symmetric polynomial of $\{\lambda_1(t),\dots,\lambda_d(t)\}\setminus\{\lambda_i(t)\}$. \end{corollary} 

\begin{proof} 
Since $T_{n-1}$ is a polynomial in its matrix argument, we have \[ T_{n-1}(P^\top D_tP)=P^\top T_{n-1}(D_t)P. \] Therefore the matrix equation in Proposition~\ref{prop:covariance-dynamics} reduces to \[ \frac{\d}{\d t}D_t = -4\left(\frac{n+1}{n!}\right)\,D_t\,T_{n-1}(D_t). \] Since $D_t$ is diagonal, so is $T_{n-1}(D_t)$, and the $i$-th diagonal entry satisfies \[ \frac{\d}{\d t}\lambda_i(t) = -4\left(\frac{n+1}{n!}\right)\,\lambda_i(t)\,[T_{n-1}(D_t)]_{ii}. \] If $n=1$, then \[ T_0(D_t)=I, \] and therefore \[ [T_0(D_t)]_{ii}=1=e_0(\widehat\lambda_i(t)). \] Thus the desired formula follows immediately. Assume now that $n\ge2$. Since \[ T_{n-1}(D_t) = \sum_{k=0}^{n-1}(-1)^k e_{n-1-k}(D_t)\,D_t^k, \] its $i$-th diagonal entry is \[ [T_{n-1}(D_t)]_{ii} = \sum_{k=0}^{n-1}(-1)^k e_{n-1-k}(D_t)\lambda_i(t)^k. \] Using \[ e_m(D_t)=e_m(\widehat\lambda_i(t))+\lambda_i(t)e_{m-1}(\widehat\lambda_i(t)) \qquad (m\ge1), \] we obtain \[ \begin{aligned} [T_{n-1}(D_t)]_{ii} &= (-1)^{n-1}\lambda_i(t)^{n-1} +\sum_{k=0}^{n-2}(-1)^k \Big( e_{n-1-k}(\widehat\lambda_i(t)) +\lambda_i(t)e_{n-2-k}(\widehat\lambda_i(t)) \Big)\lambda_i(t)^k \\ &= \sum_{k=0}^{n-1}(-1)^k e_{n-1-k}(\widehat\lambda_i(t))\lambda_i(t)^k - \sum_{k=0}^{n-2}(-1)^{k+1} e_{n-2-k}(\widehat\lambda_i(t))\lambda_i(t)^{k+1}. \end{aligned} \] The two sums telescope, leaving only the $k=0$ term: \[ [T_{n-1}(D_t)]_{ii}=e_{n-1}(\widehat\lambda_i(t)). \] Substituting this into the diagonal equation yields \[ \frac{\d}{\d t}\lambda_i(t) = -4\left(\frac{n+1}{n!}\right)\,\lambda_i(t)\,e_{n-1}(\widehat\lambda_i(t)). \] This proves the claim. \end{proof}

The preceding corollary reduces the covariance dynamics to a finite-dimensional autonomous ODE system. The next proposition shows that this reduced system is globally well-posed on the nonnegative cone.

\begin{proposition}[Global solvability of the reduced system] \label{prop:global-reduced-system} For every initial datum \[ (\lambda_1^0,\dots,\lambda_d^0)\in [0,\infty)^d, \] the system \begin{equation}\label{2.2}
\frac{\d}{\d t}\lambda_i(t) = -4\left(\frac{n+1}{n!}\right)\,\lambda_i(t)\,e_{n-1}(\widehat\lambda_i(t)), \qquad 1\le i\le d,
\end{equation} admits a unique global solution \[ (\lambda_1(t),\dots,\lambda_d(t))\in [0,\infty)^d \qquad\text{for all }t\ge0. \] Moreover, each $\lambda_i$ is nonincreasing. \end{proposition} 

\begin{proof} 
The right-hand side is a polynomial in $(\lambda_1,\dots,\lambda_d)$, hence locally Lipschitz on $\mathbb R^d$. Therefore, for every initial datum, there exists a unique maximal local solution. Let \[ \kappa_n:=4\left(\frac{n+1}{n!}\right). \] 
Since $e_{n-1}(\widehat\lambda_i(t))\ge0$ whenever $\lambda_j(t)\ge0$ for all $j$, we have \[ \frac{\d}{\d t}\lambda_i(t)=-\kappa_n\,\lambda_i(t)\,e_{n-1}(\widehat\lambda_i(t))\le0 \] as long as the solution remains in $[0,\infty)^d$. In particular, if $\lambda_i(0)\ge0$, then $\lambda_i(t)$ cannot cross zero from above, because the right-hand side vanishes at $\lambda_i=0$. Thus the nonnegative cone is positively invariant. Since each $\lambda_i$ is nonincreasing and nonnegative, we have \[ 0\le \lambda_i(t)\le \lambda_i(0) \qquad\text{for all }t \] in the interval of existence. Hence the solution remains in the compact box \[ [0,\lambda_1(0)]\times\cdots\times[0,\lambda_d(0)]. \] Because the vector field is polynomial, no finite-time blow-up is possible on a bounded set. Therefore the maximal local solution extends to all $t\ge0$. \end{proof}

Having solved the reduced system, we can now reconstruct a candidate curve in the Wasserstein space. Since the reduced system admits a global solution, the matrix-valued coefficient
\[
A_t:=2\left(\frac{n+1}{n!}\right)\,T_{n-1}(\Sigma_t)
\]
is well-defined for all $t\ge0$. Moreover, $A_t$ is continuous in $t$, because $T_{n-1}$ is a polynomial map and $t\mapsto \Sigma_t$ is continuous. Therefore, by the standard theory of linear non-autonomous ODEs, the problem
\[
\dot Q_t=-A_tQ_t,\qquad Q_0=I,
\]
admits a unique global solution. Indeed, on every finite interval $[0,T]$, Gr\"onwall's inequality yields
\[
\|Q_t\|
\le
\exp\!\left(\int_0^t \|A_s\|\,ds\right),
\qquad 0\le t\le T,
\]
and hence no finite-time blow-up can occur.

\begin{definition}[Transition map and pushforward curve] \label{def:transition-pushforward} Let $(\lambda_1(t),\dots,\lambda_d(t))$ be the global solution of \eqref{2.2}, and define \[ D_t:=\operatorname{diag}(\lambda_1(t),\dots,\lambda_d(t)), \qquad \Sigma_t:=P^\top D_tP. \] Set \[ A_t:=2\left(\frac{n+1}{n!}\right)\,T_{n-1}(\Sigma_t). \] Let $Q_t\in\mathbb R^{d\times d}$ be the unique solution of \[ \dot Q_t=-A_tQ_t, \qquad Q_0=I. \] Given $\rho_0\in\mathcal P_2(\mathbb R^d)$, define \[ X_t(x):=m_0+Q_t(x-m_0), \qquad \rho_t:=(X_t)_\#\rho_0. \] We call $(X_t)_{t\ge0}$ the transition map associated with the reduced system, and $(\rho_t)_{t\ge0}$ the pushforward curve generated by $\rho_0$. \end{definition}

The construction above is consistent with the prescribed mean and covariance dynamics. 

\begin{proposition}[Recovery of the prescribed moments] \label{prop:pushforward-recovers-moments} Let $(\rho_t)_{t\ge0}$ be the pushforward curve defined in Definition~\ref{def:transition-pushforward}. Then \[ m(\rho_t)=m_0 \qquad\text{and}\qquad \Sigma(\rho_t)=\Sigma_t \qquad\text{for all }t\ge0. \] \end{proposition} 

\begin{proof} 
By  the definition of the pushforward, \[ m(\rho_t) = \int_{\mathbb R^d}X_t(x)\,\d\rho_0(x) = \int_{\mathbb R^d}\bigl(m_0+Q_t(x-m_0)\bigr)\,\d\rho_0(x) = m_0. \] Next, \[ \begin{aligned} \Sigma(\rho_t) &= \int_{\mathbb R^d}(X_t(x)-m_0)(X_t(x)-m_0)^\top\,\d\rho_0(x) \\ &= \int_{\mathbb R^d}Q_t(x-m_0)(x-m_0)^\top Q_t^\top\,\d\rho_0(x) \\ &= Q_t\Sigma_0Q_t^\top. \end{aligned} \] Thus it remains to show that \[ Q_t\Sigma_0Q_t^\top=\Sigma_t. \] Set \[ M_t:=Q_t\Sigma_0Q_t^\top. \] Differentiating and using $\dot Q_t=-A_tQ_t$, we obtain \[ \dot M_t = -A_tM_t-M_tA_t^\top. \] Since $A_t$ is symmetric, \[ \dot M_t=-A_tM_t-M_tA_t. \] On the other hand, by the definition of $A_t$, \[ A_t=2\left(\frac{n+1}{n!}\right)T_{n-1}(\Sigma_t), \] so Proposition~\ref{prop:covariance-dynamics} implies \[ \dot\Sigma_t=-A_t\Sigma_t-\Sigma_tA_t. \] Moreover, \[ M_0=\Sigma_0. \] Therefore $M_t$ and $\Sigma_t$ solve the same linear matrix ODE with the same initial datum, and uniqueness yields \[ M_t=\Sigma_t \qquad\text{for all }t\ge0. \] Hence \[ \Sigma(\rho_t)=\Sigma_t. \] \end{proof}

Subsection~\ref{sec:3.1} has produced a candidate curve whose first two moments follow exactly the reduced dynamics. In the next subsection, we show that this curve is indeed the Wasserstein gradient flow generated by $\mathcal V_n$: more precisely, it solves the continuity equation with the explicit velocity field, satisfies the maximal slope formulation, and is uniquely determined by the initial datum.

\subsection{Weak formulation, maximal slope property, and uniqueness} In Subsection~3.1, we constructed a candidate curve by solving the reduced finite-dimensional system and pushing forward the initial datum through the associated transition map. We now show that this curve indeed realizes the Wasserstein gradient flow of $\mathcal V_n$. More precisely, we first verify that the pushforward curve is a weak solution to the continuity equation with the explicit velocity field obtained earlier. We then establish the corresponding chain rule and slope identity, which imply that the same curve is a maximal slope solution for $\mathcal V_n$. Finally, combining these identities with Proposition~\ref{prop:abstract-identification}, we prove that every maximal slope solution must be driven by the same explicit velocity field, and hence coincides with the pushforward curve. For convenience, we set \[ c_n:=\frac{n+1}{n!}, \qquad \xi_\rho(x):=2c_n\,T_{n-1}(\Sigma(\rho))(x-m(\rho)). \] Then Corollary~\ref{cor:explicit-gradient-field} reads \[ v[\rho](x)=-\xi_\rho(x). \] 

\begin{proposition}[Weak solution property of the pushforward curve] \label{prop:pushforward-weak-solution} Let $(\rho_t)_{t\ge0}$ be the pushforward curve defined in Definition~\ref{def:transition-pushforward}. Then $(\rho_t)_{t\ge0}$ is a weak solution to 

\[ \partial_t\rho_t+\nabla\cdot(\rho_t v[\rho_t])=0, \qquad \rho|_{t=0}=\rho_0. \] \end{proposition} 

\begin{proof} 

By Definition~\ref{def:transition-pushforward}, 

\[ X_t(x)=m_0+Q_t(x-m_0), \qquad \rho_t=(X_t)_\#\rho_0, \] 
where $Q_t$ solves 

\[ \dot Q_t=-A_tQ_t, \qquad A_t=2c_n\,T_{n-1}(\Sigma_t). \] 

Since $Q_t$ is $C^1$ in $t$, the map $t\mapsto X_t(x)$ is $C^1$ for every $x$, and 

\[ \dot X_t(x)=\dot Q_t(x-m_0)=-A_tQ_t(x-m_0)=-A_t(X_t(x)-m_0). \] 

By Proposition~\ref{prop:pushforward-recovers-moments}, 

\[ m(\rho_t)=m_0, \qquad \Sigma(\rho_t)=\Sigma_t. \] 

Hence 

\[ \dot X_t(x) = -2c_n\,T_{n-1}(\Sigma(\rho_t))(X_t(x)-m(\rho_t)) = v[\rho_t](X_t(x)). \]

Let $\varphi\in C_c^\infty([0,T)\times\mathbb R^d)$. Since $\rho_t=(X_t)_\#\rho_0$, we have

\[ \int_{\mathbb R^d}\varphi(t,y)\,d\rho_t(y) = \int_{\mathbb R^d}\varphi(t,X_t(x))\,\d\rho_0(x). \] 

Differentiating with respect to $t$, we obtain 

\[ \frac{\d}{\d t}\int_{\mathbb R^d}\varphi(t,X_t(x))\,\d\rho_0(x) = \int_{\mathbb R^d} \Bigl(\partial_t\varphi(t,X_t(x)) +\nabla\varphi(t,X_t(x))\cdot \dot X_t(x)\Bigr)\,\d\rho_0(x). \]

Using $\dot X_t(x)=v[\rho_t](X_t(x))$ and changing variables back to $\rho_t$, this becomes 

\[ \frac{\d}{\d t}\int_{\mathbb R^d}\varphi(t,y)\,d\rho_t(y) = \int_{\mathbb R^d} \Bigl(\partial_t\varphi(t,y)+\nabla\varphi(t,y)\cdot v[\rho_t](y)\Bigr)\,d\rho_t(y). \] 

Integrating over $[0,T]$ and using $\varphi(T,\cdot)=0$, we obtain 

\[ \int_0^T\int_{\mathbb R^d} \Bigl(\partial_t\varphi(t,y)+\nabla\varphi(t,y)\cdot v[\rho_t](y)\Bigr)\,d\rho_t(y)\,\d t + \int_{\mathbb R^d}\varphi(0,y)\,d\rho_0(y) =0. \] 

Thus $(\rho_t)$ is a weak solution in the sense of Definition~\ref{def:weak-solution}. \end{proof}

\begin{remark}
The proof of Proposition~\ref{prop:pushforward-weak-solution} shows that
$X_t$ is the characteristic flow associated with the velocity field
$v[\rho_t]$. This flow has a simple contraction property. Indeed, recall that
\[
X_t(x)=m_0+Q_t(x-m_0),
\qquad
\dot Q_t=-A_tQ_t,
\qquad
A_t=2c_nT_{n-1}(\Sigma_t).
\]
Since $A_t$ is symmetric and nonnegative, we have
\[
\frac{\d}{\d t}|X_t(x)-m_0|^2
=
-2\left\langle X_t(x)-m_0,
A_t(X_t(x)-m_0)\right\rangle
\le 0.
\]
Thus each characteristic does not move farther away from the conserved mean
$m_0$; equivalently, $|X_t(x)-m_0|$ is nonincreasing in time.

Similarly, for any $x,y\in\mathbb R^d$,
\[
\frac{\d}{\d t}|X_t(x)-X_t(y)|^2
=
-2\left\langle X_t(x)-X_t(y),
A_t(X_t(x)-X_t(y))\right\rangle
\le 0.
\]
Therefore,
\[
|X_t(x)-X_t(y)|\le |x-y|,
\qquad t\ge0.
\]
In particular, if $\rho_0$ is compactly supported, then
\[
\operatorname{diam}(\operatorname{supp}\rho_t)
\le
\operatorname{diam}(\operatorname{supp}\rho_0).
\]
This observation only means that the distance of each characteristic from
$m_0$ is nonincreasing; it does not imply that all characteristics necessarily
converge to $m_0$.
\end{remark}

The previous proposition identifies the pushforward curve as a solution to the continuity equation. In order to connect this PDE formulation with the variational framework introduced in Section \ref{sec:2.3}, we next establish a chain rule for $\mathcal V_n$ and an explicit formula for the metric slope. 

\begin{proposition}[Chain rule for $\mathcal V_n$] \label{prop:chain-rule-Vn} Let $(\mu_t)_{t\in[0,T]}\subset\mathcal P_2(\mathbb R^d)$ be a narrowly continuous weak solution to 

\[ \partial_t\mu_t+\nabla\cdot(\mu_t w_t)=0 \] 
with a Borel velocity field $w_t$ satisfying 

\[ \int_0^T \|w_t\|_{L^2(\mu_t)}^2\,\d t<\infty. \] 
Then

\[ \frac{\d}{\d t}\mathcal V_n[\mu_t] = \int_{\mathbb R^d}\langle \xi_{\mu_t}(x),w_t(x)\rangle\,d\mu_t(x) \qquad\text{for a.e. }t\in(0,T). \] 
\end{proposition} 

\begin{proof}

Set 

\[ m_t:=m(\mu_t), \qquad \Sigma_t:=\Sigma(\mu_t). \] 
Arguing exactly as in the proofs of Proposition~\ref{prop:mean-conservation} and Proposition~\ref{prop:covariance-dynamics}, but with the general velocity field $w_t$ in place of $v[\rho_t]$, one obtains

\[ \frac{\d}{\d t}m_t = \int_{\mathbb R^d}w_t(x)\,d\mu_t(x) \qquad\text{for a.e. }t\in(0,T), \] 
and 

\[ \frac{\d}{\d t}\Sigma_t = \int_{\mathbb R^d} \Bigl((x-m_t)w_t(x)^\top+w_t(x)(x-m_t)^\top\Bigr)\,d\mu_t(x) \qquad\text{for a.e. }t\in(0,T). \] 
By Proposition~\ref{prop:moment-representation}, 

\[ \mathcal V_n[\mu_t]=c_n\,e_n(\Sigma_t), \] 
and therefore Proposition~\ref{prop:differential-en} yields 

\[ \frac{\d}{\d t}\mathcal V_n[\mu_t] = c_n\,\operatorname{tr}\!\left(T_{n-1}(\Sigma_t)\frac{\d}{\d t}\Sigma_t\right). \] 
Using the symmetry of $T_{n-1}(\Sigma_t)$, we obtain 

\[ \begin{aligned} \frac{\d}{\d t}\mathcal V_n[\mu_t] &= c_n\int_{\mathbb R^d} \operatorname{tr}\Bigl( T_{n-1}(\Sigma_t)(x-m_t)w_t(x)^\top + T_{n-1}(\Sigma_t)w_t(x)(x-m_t)^\top \Bigr)\,d\mu_t(x) \\ &= 2c_n\int_{\mathbb R^d} \bigl\langle T_{n-1}(\Sigma_t)(x-m_t),\,w_t(x)\bigr\rangle\,d\mu_t(x) \\ &= \int_{\mathbb R^d}\langle \xi_{\mu_t}(x),w_t(x)\rangle\,d\mu_t(x). \end{aligned} \]
This proves the chain rule. \end{proof} 

\begin{proposition}[Slope identity for $\mathcal V_n$] \label{prop:slope-identity-Vn} For every $\rho\in\mathcal P_2(\mathbb R^d)$, 

\[ |\partial \mathcal V_n|(\rho) = \|\xi_\rho\|_{L^2(\rho)}. \]
Equivalently, 

\[ |\partial \mathcal V_n|(\rho) = 2c_n\, \left( \int_{\mathbb R^d} \bigl|T_{n-1}(\Sigma(\rho))(x-m(\rho))\bigr|^2\,\d\rho(x) \right)^{1/2}. \] \end{proposition} 

\begin{proof}
Fix $\rho\in\mathcal{P}_2(\mathbb{R}^d)$, and let us use the following notation for simplicity:

\[ m:=m(\rho), \qquad \Sigma:=\Sigma(\rho), \qquad T:=T_{n-1}(\Sigma), \qquad \xi:=\xi_\rho. \] 
We first show that 

\[ |\partial \mathcal V_n|(\rho)\le \|\xi\|_{L^2(\rho)}. \] 
Let $\nu\in\mathcal P_2(\mathbb R^d)$, and let $\pi\in\Gamma_o(\rho,\nu)$, where $\Gamma_o(\rho,\nu)$ denotes the set of all optimal plans from $\rho$ to $\nu$. Write 

\[ \widetilde m:=m(\nu), \qquad \widetilde\Sigma:=\Sigma(\nu). \] 
Since $e_n$ is a polynomial on the finite-dimensional space of symmetric matrices, its differential is locally Lipschitz. Hence, for $\nu$ sufficiently close to $\rho$ in $W_2$, there exists a constant $C_\rho>0$ such that 

\[ \bigl|e_n(\widetilde\Sigma)-e_n(\Sigma)-\operatorname{tr}(T(\widetilde\Sigma-\Sigma))\bigr| \le C_\rho |\widetilde\Sigma-\Sigma|^2. \] 
Moreover, because $W_2(\rho,\nu)\to0$ implies local boundedness of the second moments, we have 

\[ |\widetilde\Sigma-\Sigma| \le C_\rho W_2(\rho,\nu), \] 
and therefore 

\[ \mathcal V_n[\nu]-\mathcal V_n[\rho] = c_n\,\operatorname{tr}(T(\widetilde\Sigma-\Sigma)) + O\!\left(W_2(\rho,\nu)^2\right). \] 
Next, define 

\[ h(x,y):=(y-x)-(\widetilde m-m). \] 
Then

\[ \int_{\bbr^\dm\times\bbr^\dm} h(x,y)\,d\pi(x,y)=0, \] and 

\[ \widetilde\Sigma-\Sigma = \int_{\bbr^\dm\times\bbr^\dm} \Bigl((x-m)h^\top+h(x-m)^\top+hh^\top\Bigr)\,d\pi(x,y). \] 
Hence 
\[ \begin{aligned} \operatorname{tr}(T(\widetilde\Sigma-\Sigma)) &= 2\int_{\bbr^\dm\times\bbr^\dm} \langle T(x-m),h(x,y)\rangle\,d\pi(x,y) + \int_{\bbr^\dm\times\bbr^\dm} h(x,y)^\top T h(x,y)\,d\pi(x,y) \\ &= 2\int_{\bbr^\dm\times\bbr^\dm} \langle T(x-m),y-x\rangle\,d\pi(x,y) + O\!\left(W_2(\rho,\nu)^2\right), \end{aligned} \] 
because 

\[ \int_{\bbr^\dm} (x-m)\,\d\rho(x)=0 \] 
and 

\[ \int_{\bbr^\dm\times\bbr^\dm} |h(x,y)|^2\,d\pi(x,y)\le C\,W_2(\rho,\nu)^2. \] 
Combining the last two displays, we find

\[ \mathcal V_n[\rho]-\mathcal V_n[\nu] \le \int_{\mathbb R^d\times\mathbb R^d} \langle \xi(x),x-y\rangle\,d\pi(x,y) + O\!\left(W_2(\rho,\nu)^2\right). \] 
By the Cauchy--Schwarz inequality,

\[ \int_{\bbr^\dm\times\bbr^\dm} \langle \xi(x),x-y\rangle\,d\pi(x,y) \le \|\xi\|_{L^2(\rho)}\,W_2(\rho,\nu). \] 
Therefore, 

\[ \frac{(\mathcal V_n[\rho]-\mathcal V_n[\nu])_+}{W_2(\rho,\nu)} \le \|\xi\|_{L^2(\rho)}+O\!\left(W_2(\rho,\nu)\right), \] 
and taking $W_2(\rho,\nu)\to 0$ yields 

\[ |\partial \mathcal V_n|(\rho)\le \|\xi\|_{L^2(\rho)}. \] 
We now prove the reverse inequality. If $\xi=0$, there is nothing to prove. Assume $\xi\neq0$. Since $C_c^\infty(\mathbb R^d;\mathbb R^d)$ is dense in $L^2(\rho;\mathbb R^d)$, we may choose $r_k\in C_c^\infty(\mathbb R^d;\mathbb R^d)$ such that

\[ r_k\to \xi \qquad\text{in }L^2(\rho;\mathbb R^d). \] 
For $\varepsilon>0$, define 

\[ S_{\varepsilon,k}(x):=x-\varepsilon r_k(x), \qquad \rho_{\varepsilon,k}:=(S_{\varepsilon,k})_\#\rho. \] 
By Proposition~\ref{prop:first-variation-Vn}, 

\[ \lim_{\varepsilon\downarrow0} \frac{\mathcal V_n[\rho]-\mathcal V_n[\rho_{\varepsilon,k}]}{\varepsilon} = \int_{\mathbb R^d}\langle \xi(x),r_k(x)\rangle\,\d\rho(x). \] 
On the other hand, 

\[ W_2(\rho,\rho_{\varepsilon,k}) \le \left(\int_{\mathbb R^d}|x-S_{\varepsilon,k}(x)|^2\,\d\rho(x)\right)^{1/2} = \varepsilon \|r_k\|_{L^2(\rho)}. \] 
Therefore, 

\[ |\partial \mathcal V_n|(\rho) \ge \limsup_{\varepsilon\downarrow0} \frac{\mathcal V_n[\rho]-\mathcal V_n[\rho_{\varepsilon,k}]}{W_2(\rho,\rho_{\varepsilon,k})} \ge \frac{\int \langle \xi,r_k\rangle\,d\rho}{\|r_k\|_{L^2(\rho)}}. \] 
Letting $k\to\infty$, we obtain \[ |\partial \mathcal V_n|(\rho)\ge \|\xi\|_{L^2(\rho)}. \] 
Combining the two inequalities proves the claim. 

\end{proof} 

The slope identity obtained above shows that the vector field
$\xi_\rho$ is not merely a candidate first variation, but actually realizes
the metric slope of $\mathcal V_n$. In order to use this information in the
maximal-slope formulation, we next record a direct consequence of the chain
rule and the slope identity: the local slope of $\mathcal V_n$ is a strong
upper gradient. This provides the missing link between the explicit
continuity-equation formulation and the variational gradient-flow
formulation.

\begin{lemma}[Strong upper gradient property]
\label{lem:strong-upper-gradient-Vn}
The local slope $|\partial \mathcal V_n|$ is a strong upper gradient for
$\mathcal V_n$ on $\mathcal P_2(\mathbb R^d)$. More precisely, for every
curve $(\mu_t)_{t\in[0,T]}\in AC^2(0,T;\mathcal P_2(\mathbb R^d))$, the map
$t\mapsto \mathcal V_n[\mu_t]$ is absolutely continuous and satisfies
\[
\left|\frac{\d}{\d t}\mathcal V_n[\mu_t]\right|
\le
|\partial \mathcal V_n|(\mu_t)\,|\mu'|(t)
\qquad\text{for a.e. }t\in(0,T).
\]
Consequently, for all $0\le s\le t\le T$,
\[
|\mathcal V_n[\mu_t]-\mathcal V_n[\mu_s]|
\le
\int_s^t
|\partial \mathcal V_n|(\mu_r)\,|\mu'|(r)\,\d r .
\]
\end{lemma}

\begin{proof}
Let $(\mu_t)_{t\in[0,T]}\in AC^2(0,T;\mathcal P_2(\mathbb R^d))$. By the
continuity-equation representation of absolutely continuous curves in
$\mathcal P_2(\mathbb R^d)$, there exists a Borel vector field
$w_t\in L^2(\mu_t;\mathbb R^d)$ such that
\[
\partial_t\mu_t+\nabla\cdot(\mu_t w_t)=0
\quad\text{in }\mathcal D'((0,T)\times\mathbb R^d),
\]
and
\[
\|w_t\|_{L^2(\mu_t)}\le |\mu'|(t)
\qquad\text{for a.e. }t\in(0,T).
\]
By the chain rule for $\mathcal V_n$ established in
Proposition~\ref{prop:chain-rule-Vn}, we have
\[
\frac{\d}{\d t}\mathcal V_n[\mu_t]
=
\int_{\mathbb R^d}
\langle \xi_{\mu_t}(x),w_t(x)\rangle\,d\mu_t(x)
\qquad\text{for a.e. }t\in(0,T).
\]
Hence, by the Cauchy--Schwarz inequality,
\[
\left|\frac{\d}{\d t}\mathcal V_n[\mu_t]\right|
\le
\|\xi_{\mu_t}\|_{L^2(\mu_t)}
\|w_t\|_{L^2(\mu_t)}.
\]
Using the slope identity from Proposition~\ref{prop:slope-identity-Vn},
\[
\|\xi_{\mu_t}\|_{L^2(\mu_t)}
=
|\partial \mathcal V_n|(\mu_t),
\]
we obtain
\[
\left|\frac{\d}{\d t}\mathcal V_n[\mu_t]\right|
\le
|\partial \mathcal V_n|(\mu_t)\,|\mu'|(t)
\qquad\text{for a.e. }t\in(0,T).
\]
In particular, $t\mapsto \mathcal V_n[\mu_t]$ is absolutely continuous, and
integrating the preceding differential inequality over $[s,t]$ gives
\[
|\mathcal V_n[\mu_t]-\mathcal V_n[\mu_s]|
\le
\int_s^t
|\partial \mathcal V_n|(\mu_r)\,|\mu'|(r)\,\d r .
\]
This proves the strong upper gradient property.
\end{proof}

We can now verify that the explicitly constructed pushforward curve is not
only a weak solution of the continuity equation, but also the Wasserstein
gradient flow of $\mathcal V_n$ in the variational sense. Indeed, along this
curve the velocity field is precisely $-\xi_{\rho_t}$, and the slope identity
shows that the $L^2(\rho_t)$-norm of this velocity coincides with the metric
slope. Together with the strong upper gradient property, this yields the
energy dissipation identity below.

\begin{theorem}[Energy dissipation identity and maximal slope property]
\label{thm:maximal-slope-property}
Let $(\rho_t)_{t\ge0}$ be the pushforward curve defined in
Definition~\ref{def:transition-pushforward}. Then $(\rho_t)_{t\ge0}$ is a
maximal slope solution for $\mathcal V_n$ with respect to
$g=|\partial\mathcal V_n|$. More precisely, for every $0\le s\le t<\infty$,
\[
\mathcal V_n[\rho_t]
+
\frac12\int_s^t |\rho'|^2(r)\,\d r
+
\frac12\int_s^t |\partial\mathcal V_n|^2(\rho_r)\,\d r
=
\mathcal V_n[\rho_s].
\]
\end{theorem}

\begin{proof}
Fix $T>0$. By Proposition~\ref{prop:pushforward-weak-solution}, the
pushforward curve $(\rho_t)_{t\ge0}$ is a weak solution of
\[
\partial_t\rho_t+\nabla\cdot(\rho_t v[\rho_t])=0,
\qquad
\rho|_{t=0}=\rho_0,
\]
where
\[
v[\rho_t](x)=-\xi_{\rho_t}(x).
\]
Moreover, by Proposition~\ref{prop:pushforward-recovers-moments},
\[
m(\rho_t)=m_0,
\qquad
\Sigma(\rho_t)=\Sigma_t .
\]
Hence, with
\[
c_n:=\frac{n+1}{n!},
\]
we have
\[
v[\rho_t](x)
=
-2c_nT_{n-1}(\Sigma_t)(x-m_0).
\]

We first verify that the velocity field is square-integrable along the curve
on every finite time interval. Since $t\mapsto\Sigma_t$ is continuous and
$T_{n-1}$ is a polynomial map, $t\mapsto T_{n-1}(\Sigma_t)$ is continuous.
Furthermore,
\[
\begin{aligned}
\|v[\rho_t]\|_{L^2(\rho_t)}^2
&=
4c_n^2
\int_{\mathbb R^d}
\left|T_{n-1}(\Sigma_t)(x-m_0)\right|^2\,\d\rho_t(x) \\
&=
4c_n^2\,
\operatorname{tr}
\left(
T_{n-1}(\Sigma_t)^2\Sigma_t
\right).
\end{aligned}
\]
The right-hand side is continuous in $t$ on $[0,T]$. Therefore
\[
\int_0^T
\|v[\rho_t]\|_{L^2(\rho_t)}^2\,\d t
<\infty .
\]
By the converse part of Proposition~\ref{prop:continuity-equation}, it follows that
\[
\rho\in AC^2(0,T;\mathcal P_2(\mathbb R^d))
\]
and
\[
|\rho'|(t)\le \|v[\rho_t]\|_{L^2(\rho_t)}
\qquad\text{for a.e. }t\in(0,T).
\]
Since $v[\rho_t]=-\xi_{\rho_t}$, Proposition~\ref{prop:slope-identity-Vn} gives
\[
\|v[\rho_t]\|_{L^2(\rho_t)}
=
\|\xi_{\rho_t}\|_{L^2(\rho_t)}
=
|\partial\mathcal V_n|(\rho_t).
\]
Thus
\[
|\rho'|(t)
\le
|\partial\mathcal V_n|(\rho_t)
\qquad\text{for a.e. }t\in(0,T).
\]

Next, applying the chain rule in Proposition~\ref{prop:chain-rule-Vn}, we obtain
\[
\frac{\d}{\d t}\mathcal V_n[\rho_t]
=
\int_{\mathbb R^d}
\langle \xi_{\rho_t}(x),v[\rho_t](x)\rangle\,\d\rho_t(x)
\qquad\text{for a.e. }t\in(0,T).
\]
Since $v[\rho_t]=-\xi_{\rho_t}$, this becomes
\[
\frac{\d}{\d t}\mathcal V_n[\rho_t]
=
-\|\xi_{\rho_t}\|_{L^2(\rho_t)}^2.
\]
Using again Proposition~\ref{prop:slope-identity-Vn}, we get
\begin{equation}\label{3.6}
    \frac{\d}{\d t}\mathcal V_n[\rho_t]
=
-|\partial\mathcal V_n|^2(\rho_t)
\qquad\text{for a.e. }t\in(0,T).
\end{equation}
In particular, $t\mapsto\mathcal V_n[\rho_t]$ is nonincreasing.

We now use Lemma~\ref{lem:strong-upper-gradient-Vn}. Since
$|\partial\mathcal V_n|$ is a strong upper gradient for $\mathcal V_n$, we have
\[
\left|
\frac{\d}{\d t}\mathcal V_n[\rho_t]
\right|
\le
|\partial\mathcal V_n|(\rho_t)\,|\rho'|(t)
\qquad\text{for a.e. }t\in(0,T),
\]
which implies that
\[
|\partial\mathcal V_n|^2(\rho_t)
\le
|\partial\mathcal V_n|(\rho_t)\,|\rho'|(t)
\qquad\text{for a.e. }t\in(0,T)
\]
 using \eqref{3.6}. Combining this with
\[
|\rho'|(t)\le |\partial\mathcal V_n|(\rho_t),
\]
we conclude that
\[
|\rho'|(t)=|\partial\mathcal V_n|(\rho_t)
\qquad\text{for a.e. }t\in(0,T).
\]

Consequently,
\[
\frac{\d}{\d t}\mathcal V_n[\rho_t]
=
-|\partial\mathcal V_n|^2(\rho_t)
=
-|\rho'|^2(t)
\qquad\text{for a.e. }t\in(0,T).
\]
Integrating this identity over $[s,t]\subset[0,T]$, we obtain
\[
\mathcal V_n[\rho_t]-\mathcal V_n[\rho_s]
=
-\int_s^t |\partial\mathcal V_n|^2(\rho_r)\,\d r
=
-\int_s^t |\rho'|^2(r)\,\d r.
\]
Since
\[
|\rho'|(r)=|\partial\mathcal V_n|(\rho_r)
\qquad\text{for a.e. }r,
\]
the previous identity is equivalently written as
\[
\mathcal V_n[\rho_t]
+
\frac12\int_s^t |\rho'|^2(r)\,\d r
+
\frac12\int_s^t |\partial\mathcal V_n|^2(\rho_r)\,\d r
=
\mathcal V_n[\rho_s].
\]
This proves the energy dissipation identity.

In particular, the energy dissipation inequality holds. Therefore
$(\rho_t)_{t\ge0}$ is a maximal slope solution for $\mathcal V_n$ with respect
to $g=|\partial\mathcal V_n|$.
\end{proof}

We now prove uniqueness in a stronger PDE sense. The argument follows the
standard $W_2$-stability estimate for continuity equations, but we avoid any
compact-support assumption by using the affine representation of the velocity
field. The key point is that the dependence of $v[\rho]$ on $\rho$ is entirely
through the mean $m(\rho)$ and the covariance matrix $\Sigma(\rho)$, both of
which are locally Lipschitz with respect to $W_2$ on sets with bounded second
moments.

\begin{lemma}[Moment and covariance stability]\label{lem:moment-covariance-stability}
Let $\rho,\sigma\in\mathcal P_2(\mathbb R^d)$ and suppose that
\[
M_2(\rho)+M_2(\sigma)\le R
\]
for some $R>0$. Then there exists a constant $C_R>0$, depending only on $R$
and $d$, such that
\[
|m(\rho)-m(\sigma)|\le W_2(\rho,\sigma),
\]
and
\[
\|\Sigma(\rho)-\Sigma(\sigma)\|_{\mathrm F}
\le C_R W_2(\rho,\sigma).
\]
\end{lemma}

\begin{proof}
Let $\pi$ be an optimal transport plan from $\rho$ to $\sigma$. Then
\[
m(\rho)-m(\sigma)
=
\int_{\mathbb R^d\times\mathbb R^d}(x-y)\,d\pi(x,y),
\]
and hence, by Cauchy--Schwarz,
\[
|m(\rho)-m(\sigma)|
\le
\left(\int_{\mathbb R^d\times\mathbb R^d}|x-y|^2\,d\pi(x,y)\right)^{1/2}
=
W_2(\rho,\sigma).
\]

For the covariance matrices, set
\[
m_\rho:=m(\rho),\qquad m_\sigma:=m(\sigma),
\]
and write
\[
a:=x-m_\rho,\qquad b:=y-m_\sigma.
\]
Then
\[
\Sigma(\rho)-\Sigma(\sigma)
=
\int_{\mathbb R^d\times\mathbb R^d}
\left(aa^\top-bb^\top\right)\,d\pi(x,y).
\]
Using
\[
\|aa^\top-bb^\top\|_{\mathrm F}
\le
|a-b|(|a|+|b|),
\]
we obtain
\[
\begin{aligned}
\|\Sigma(\rho)-\Sigma(\sigma)\|_{\mathrm F}
&\le
\int_{\bbr^\dm\times\bbr^\dm} |a-b|(|a|+|b|)\,d\pi  \\
&\le
\left(\int_{\bbr^\dm\times\bbr^\dm} |a-b|^2\,d\pi\right)^{1/2}
\left(\int_{\bbr^\dm\times\bbr^\dm} (|a|+|b|)^2\,d\pi\right)^{1/2}.
\end{aligned}
\]
Since
\[
a-b=(x-y)-(m_\rho-m_\sigma),
\]
we have
\[
\int_{\bbr^\dm\times\bbr^\dm} |a-b|^2\,d\pi
\le
2\int_{\bbr^\dm\times\bbr^\dm} |x-y|^2\,d\pi
+
2|m_\rho-m_\sigma|^2
\le
4W_2^2(\rho,\sigma).
\]
Moreover,
\begin{align*}
\int_{\bbr^\dm\times\bbr^\dm} (|a|+|b|)^2\,d\pi
&\le
2\int_{\bbr^\dm\times\bbr^\dm} |a|^2\,d\pi
+
2\int_{\bbr^\dm\times\bbr^\dm} |b|^2\,d\pi\\
&=
2(M_2(\rho)-|m_\rho|^2)
+
2(M_2(\sigma)-|m_\sigma|^2)\\
&\le 2R.
\end{align*}
Combining the previous estimates gives
\[
\|\Sigma(\rho)-\Sigma(\sigma)\|_{\mathrm F}
\le
2\sqrt{2R}\,W_2(\rho,\sigma).
\]
This proves the claim.
\end{proof}

\begin{lemma}[Local $W_2$-Lipschitz estimate for the velocity field]
\label{lem:velocity-W2-lipschitz-P2}
Let $\rho,\sigma\in\mathcal P_2(\mathbb R^d)$ and suppose that
\[
M_2(\rho)+M_2(\sigma)\le R.
\]
Then there exists a constant $C_R>0$, depending only on $n,d$ and $R$, such that
\[
\|v[\rho]-v[\sigma]\|_{L^2(\sigma)}
\le
C_R W_2(\rho,\sigma).
\]
\end{lemma}

\begin{proof}
Recall that
\[
v[\rho](x)
=
-A_\rho(x-m_\rho),
\qquad
A_\rho:=
2c_nT_{n-1}(\Sigma(\rho)),
\qquad
c_n:=\frac{n+1}{n!}.
\]
Similarly,
\[
v[\sigma](x)
=
-A_\sigma(x-m_\sigma).
\]
For $y\in\mathbb R^d$, we decompose
\[
\begin{aligned}
v[\rho](y)-v[\sigma](y)
&=
-A_\rho(y-m_\rho)+A_\sigma(y-m_\sigma) \\
&=
-A_\rho\bigl((y-m_\sigma)+(m_\sigma-m_\rho)\bigr)
+A_\sigma(y-m_\sigma) \\
&=
-(A_\rho-A_\sigma)(y-m_\sigma)
-A_\rho(m_\sigma-m_\rho).
\end{aligned}
\]
Therefore, using the operator norm for matrices acting on vectors,
\[
\begin{aligned}
\|v[\rho]-v[\sigma]\|_{L^2(\sigma)}
&\le
\|A_\rho-A_\sigma\|_{\mathrm{op}}
\|y-m_\sigma\|_{L^2(\sigma)}
+
\|A_\rho\|_{\mathrm{op}}|m_\rho-m_\sigma|.
\end{aligned}
\]

We now estimate the two matrix terms. Since $T_{n-1}$ is a polynomial map on
the finite-dimensional space of symmetric matrices, it is Lipschitz on bounded
sets. Moreover,
\[
\|\Sigma(\rho)\|_{\mathrm F}
\le
\operatorname{tr}\Sigma(\rho)
\le
M_2(\rho),
\qquad
\|\Sigma(\sigma)\|_{\mathrm F}
\le
\operatorname{tr}\Sigma(\sigma)
\le
M_2(\sigma).
\]
Thus, under the assumption
\[
M_2(\rho)+M_2(\sigma)\le R,
\]
the matrices $\Sigma(\rho)$ and $\Sigma(\sigma)$ remain in a bounded subset of
the space of  symmetric matrices. Hence there exists a constant $C_R>0$, depending only
on $n,d$ and $R$, such that
\[
\|T_{n-1}(\Sigma(\rho))-T_{n-1}(\Sigma(\sigma))\|_{\mathrm F}
\le
C_R\|\Sigma(\rho)-\Sigma(\sigma)\|_{\mathrm F}.
\]
Since $\|B\|_{\mathrm{op}}\le \|B\|_{\mathrm F}$ for every matrix $B$, we get
\[
\begin{aligned}
\|A_\rho-A_\sigma\|_{\mathrm{op}}
&\le
\|A_\rho-A_\sigma\|_{\mathrm F} \\
&=
2c_n
\|T_{n-1}(\Sigma(\rho))-T_{n-1}(\Sigma(\sigma))\|_{\mathrm F} \\
&\le
C_R\|\Sigma(\rho)-\Sigma(\sigma)\|_{\mathrm F}.
\end{aligned}
\]
Similarly, since $\Sigma(\rho)$ is bounded in Frobenius norm,
\[
\|A_\rho\|_{\mathrm{op}}
\le
\|A_\rho\|_{\mathrm F}
=
2c_n\|T_{n-1}(\Sigma(\rho))\|_{\mathrm F}
\le
C_R.
\]

By Lemma~\ref{lem:moment-covariance-stability},
\[
\|\Sigma(\rho)-\Sigma(\sigma)\|_{\mathrm F}
\le
C_R W_2(\rho,\sigma),
\qquad
|m_\rho-m_\sigma|
\le
W_2(\rho,\sigma).
\]
Furthermore,
\[
\|y-m_\sigma\|_{L^2(\sigma)}
=
\left(
\int_{\mathbb R^d}|y-m_\sigma|^2\,d\sigma(y)
\right)^{1/2}
=
\left(\operatorname{tr}\Sigma(\sigma)\right)^{1/2}
\le
R^{1/2}.
\]
Combining the preceding estimates, we obtain
\[
\begin{aligned}
\|v[\rho]-v[\sigma]\|_{L^2(\sigma)}
&\le
C_R\|\Sigma(\rho)-\Sigma(\sigma)\|_{\mathrm F}
\left(\operatorname{tr}\Sigma(\sigma)\right)^{1/2}
+
C_R|m_\rho-m_\sigma| \\
&\le
C_R W_2(\rho,\sigma).
\end{aligned}
\]
Here the constant $C_R$ may change from line to line, but depends only on
$n,d$ and $R$. This proves the claim.
\end{proof}

\begin{lemma}[Differential estimate for the squared Wasserstein distance]
\label{lem:W2-differential-estimate}
Let $(\rho_t)_{t\in[0,T]}$ and $(\sigma_t)_{t\in[0,T]}$ be curves in
$AC^2(0,T;\mathcal P_2(\mathbb R^d))$ satisfying
\[
\partial_t\rho_t+\nabla\cdot(\rho_t v_t)=0,
\qquad
\partial_t\sigma_t+\nabla\cdot(\sigma_t u_t)=0
\]
in the sense of distributions, where
\[
v_t\in L^2(\rho_t;\mathbb R^d),
\qquad
u_t\in L^2(\sigma_t;\mathbb R^d).
\]
Assume moreover that $v_t$ and $u_t$ are the tangent velocity fields of
$\rho_t$ and $\sigma_t$, respectively. Then, for
\[
F(t):=\frac12 W_2^2(\rho_t,\sigma_t),
\]
we have, for a.e. $t\in(0,T)$,
\[
F'(t)
\le
\int_{\mathbb R^{d}\times\bbr^\dm}
\langle x-y, v_t(x)-u_t(y)\rangle\,d\pi_t(x,y)
\]
for every $\pi_t\in\Gamma_o(\rho_t,\sigma_t)$.
\end{lemma}

\begin{proof}
Since $\rho,\sigma\in AC^2(0,T;\mathcal P_2(\mathbb R^d))$, the map
\[
t\mapsto W_2(\rho_t,\sigma_t)
\]
is absolutely continuous. Hence
\[
F(t):=\frac12W_2^2(\rho_t,\sigma_t)
\]
is also absolutely continuous and differentiable for a.e. $t\in(0,T)$.

Fix a time $t\in(0,T)$ such that $F$ is differentiable at $t$ and the
first-order approximation property in
Proposition~\ref{prop:first-order-approximation} holds for both curves
$\rho_t$ and $\sigma_t$. Since $v_t$ and $u_t$ are the tangent velocity fields
of $\rho_t$ and $\sigma_t$, respectively, Proposition~\ref{prop:first-order-approximation}
gives
\[
W_2\bigl(\rho_{t+h},(\mathrm{Id}+hv_t)_\#\rho_t\bigr)=o(|h|)
\]
and
\[
W_2\bigl(\sigma_{t+h},(\mathrm{Id}+hu_t)_\#\sigma_t\bigr)=o(|h|)
\]
as $h\to0$. The set of such times has full measure in $(0,T)$.

Set
\[
\widetilde\rho_{t,h}:=(\mathrm{Id}+hv_t)_\#\rho_t,
\qquad
\widetilde\sigma_{t,h}:=(\mathrm{Id}+hu_t)_\#\sigma_t .
\]
Let $\pi_t\in\Gamma_o(\rho_t,\sigma_t)$ be arbitrary. Then
\[
\widetilde\pi_{t,h}
:=
(\mathrm{Id}+hv_t,\mathrm{Id}+hu_t)_\#\pi_t
\]
is a coupling between $\widetilde\rho_{t,h}$ and
$\widetilde\sigma_{t,h}$. Therefore,
\[
\begin{aligned}
W_2^2(\widetilde\rho_{t,h},\widetilde\sigma_{t,h})
&\le
\int_{\mathbb R^{d}\times\bbr^\dm}
\left|
(x+hv_t(x))-(y+hu_t(y))
\right|^2
\,d\pi_t(x,y) \\
&=
\int_{\mathbb R^{d}\times\bbr^\dm}
|x-y|^2\,d\pi_t(x,y) \\
&\quad
+
2h
\int_{\mathbb R^{d}\times\bbr^\dm}
\langle x-y,v_t(x)-u_t(y)\rangle
\,d\pi_t(x,y) \\
&\quad
+
h^2
\int_{\mathbb R^{d}\times\bbr^\dm}
|v_t(x)-u_t(y)|^2\,d\pi_t(x,y).
\end{aligned}
\]
Since $\pi_t$ is optimal,
\[
\int_{\mathbb R^{d}\times\bbr^\dm}|x-y|^2\,d\pi_t(x,y)
=
W_2^2(\rho_t,\sigma_t).
\]
Moreover,
\[
\int_{\mathbb R^{d}\times\bbr^\dm}
|v_t(x)-u_t(y)|^2\,d\pi_t(x,y)
<\infty,
\]
because $v_t\in L^2(\rho_t)$ and $u_t\in L^2(\sigma_t)$.

By the triangle inequality and the first-order approximation property,
\[
\begin{aligned}
W_2(\rho_{t+h},\sigma_{t+h})
&\le
W_2(\rho_{t+h},\widetilde\rho_{t,h})
+
W_2(\widetilde\rho_{t,h},\widetilde\sigma_{t,h})
+
W_2(\widetilde\sigma_{t,h},\sigma_{t+h}) \\
&=
W_2(\widetilde\rho_{t,h},\widetilde\sigma_{t,h})
+
o(|h|).
\end{aligned}
\]
Since $W_2(\widetilde\rho_{t,h},\widetilde\sigma_{t,h})$ remains bounded as
$h\to0$, it follows that
\[
W_2^2(\rho_{t+h},\sigma_{t+h})
\le
W_2^2(\widetilde\rho_{t,h},\widetilde\sigma_{t,h})
+
o(|h|).
\]
Combining the previous estimates, we obtain
\[
\begin{aligned}
F(t+h)-F(t)
&=
\frac12W_2^2(\rho_{t+h},\sigma_{t+h})
-
\frac12W_2^2(\rho_t,\sigma_t) \\
&\le
h
\int_{\mathbb R^{d}\times\bbr^\dm}
\langle x-y,v_t(x)-u_t(y)\rangle
\,d\pi_t(x,y) \\
&\quad
+
\frac{h^2}{2}
\int_{\mathbb R^{d}\times\bbr^\dm}
|v_t(x)-u_t(y)|^2\,d\pi_t(x,y)
+
o(|h|).
\end{aligned}
\]
Dividing by $h>0$ and letting $h\downarrow0$, we get
\[
F'(t)
\le
\int_{\mathbb R^{d}\times\bbr^\dm}
\langle x-y,v_t(x)-u_t(y)\rangle
\,d\pi_t(x,y).
\]
Since the chosen time $t$ belongs to a full-measure subset of $(0,T)$, the
estimate holds for a.e. $t\in(0,T)$. Since the optimal plan $\pi_t\in
\Gamma_o(\rho_t,\sigma_t)$ was arbitrary, so the estimate holds for every
such optimal plan.
\end{proof}

Theorem~\ref{thm:maximal-slope-property} shows that the pushforward curve
constructed above is a maximal slope solution for $\mathcal V_n$. For uniqueness
in the variational formulation, we must also show the converse direction: any
maximal slope solution is necessarily driven by the explicit velocity field
obtained from the first variation of $\mathcal V_n$. This follows from the
abstract velocity identification in Proposition~\ref{prop:abstract-identification},
together with the chain rule and the slope identity established above.

\begin{corollary}[Identification of maximal slope solutions]
	\label{cor:maximal-slope-identification-Vn}
	Let $(\rho_t)_{t\in[0,T]}$ be a maximal slope solution for $\mathcal V_n$
	with respect to $|\partial \mathcal V_n|$. Let $w_t$ be the minimal velocity
	field associated with $(\rho_t)$, so that
	\[
	\partial_t\rho_t+\nabla\cdot(\rho_t w_t)=0
	\qquad\text{in }\mathcal D'((0,T)\times\mathbb R^d),
	\]
	and
	\[
	\|w_t\|_{L^2(\rho_t)}=|\rho'|(t)
	\qquad\text{for a.e. }t\in(0,T).
	\]
	Then
	\[
	w_t=-\xi_{\rho_t}
	\qquad\text{for a.e. }t\in(0,T).
	\]
	In particular, $(\rho_t)$ is a weak solution to the explicit continuity
	equation
	\[
	\partial_t\rho_t+\nabla\cdot(\rho_t v[\rho_t])=0,
	\qquad
	v[\rho](x):=-\xi_\rho(x).
	\]
\end{corollary}

\begin{proof}
	Since $(\rho_t)$ is a maximal slope solution, we have
	\[
	\rho\in AC^2(0,T;(\mathcal P_2(\mathbb R^d),W_2)).
	\]
	By Lemma~\ref{lem:strong-upper-gradient-Vn}, the local slope
	$|\partial \mathcal V_n|$ is a strong upper gradient for $\mathcal V_n$.
	Let $\varphi$ be the nonincreasing representative appearing in
	Definition~\ref{def:maximal-slope}. Then
	\[
	\varphi(t)=\mathcal V_n[\rho_t]
	\qquad\text{for a.e. }t\in(0,T).
	\]
	On the other hand, by Lemma~\ref{lem:Vn-W2-continuity} and the
	$W_2$-continuity of $\rho$, the map
	\[
	t\longmapsto \mathcal V_n[\rho_t]
	\]
	is continuous on $[0,T]$. Hence, by the monotonicity of $\varphi$, we may
	identify $\varphi$ with $\mathcal V_n[\rho_t]$ on $(0,T)$, up to a redefinition at the endpoints.

	By Proposition~\ref{prop:chain-rule-Vn}, applied to the continuity equation
	with the minimal velocity field $w_t$, the map
	$t\mapsto \mathcal V_n[\rho_t]$ is absolutely continuous and satisfies
	\[
	\varphi'(t)
	=
	\frac{\d}{\d t}\mathcal V_n[\rho_t]
	=
	\int_{\mathbb R^d}
	\langle \xi_{\rho_t}(x),w_t(x)\rangle\,\d\rho_t(x)
	\qquad\text{for a.e. }t\in(0,T).
	\]
	Moreover, Proposition~\ref{prop:slope-identity-Vn} gives the slope identity
	\[
	\|\xi_{\rho_t}\|_{L^2(\rho_t)}
	=
	|\partial \mathcal V_n|(\rho_t)
	\qquad\text{for a.e. }t\in(0,T).
	\]
	Therefore Proposition~\ref{prop:abstract-identification}, applied with
	$F=\mathcal V_n$, yields
	\[
	w_t=-\xi_{\rho_t}
	\qquad\text{for a.e. }t\in(0,T).
	\]
	Since $v[\rho]=-\xi_\rho$, the continuity equation satisfied by $w_t$
	becomes
	\[
	\partial_t\rho_t+\nabla\cdot(\rho_t v[\rho_t])=0
	\qquad\text{in }\mathcal D'((0,T)\times\mathbb R^d).
	\]
	This proves the claim.
\end{proof}

The preceding corollary reduces the uniqueness of maximal slope solutions to
the uniqueness of weak solutions for the explicit continuity equation. We now
prove this PDE uniqueness through a $W_2$-stability estimate.

\begin{theorem}[$W_2$-stability and uniqueness]\label{thm:W2-stability-uniqueness}
Let $(\rho_t)_{t\in[0,T]}$ and $(\sigma_t)_{t\in[0,T]}$ be two weak solutions in
$AC^2(0,T;\mathcal P_2(\mathbb R^d))$ of
\[
\partial_t\mu_t+\nabla\cdot(\mu_t v[\mu_t])=0,
\]
where
\[
v[\mu](x)
=
-2c_nT_{n-1}(\Sigma(\mu))(x-m(\mu)),
\qquad
c_n:=\frac{n+1}{n!}.
\]
Then there exists a constant $C_T>0$ such that
\[
W_2(\rho_t,\sigma_t)
\le
e^{C_Tt}W_2(\rho_0,\sigma_0),
\qquad
0\le t\le T.
\]
In particular, if $\rho_0=\sigma_0$, then
\[
\rho_t=\sigma_t
\qquad\emph{for all }t\in[0,T].
\]
\end{theorem}

\begin{proof}
Since $\rho,\sigma\in AC^2(0,T;\mathcal P_2(\mathbb R^d))$, the curves are
continuous with respect to $W_2$. In particular, by the $W_2$-continuity of
the second moment, we have
\[
R_T:=
\sup_{0\le t\le T}
\left(M_2(\rho_t)+M_2(\sigma_t)\right)
<\infty.
\]

Set
\[
F(t):=\frac12W_2^2(\rho_t,\sigma_t).
\]
Since the prescribed velocity field is given by
\[
v[\mu]=-\nabla\left(\frac{\delta\mathcal V_n}{\delta\mu}\right),
\]
it belongs to the Wasserstein tangent space at $\mu$. Hence, for the weak
solutions under consideration, $v[\rho_t]$ and $v[\sigma_t]$ are the tangent
velocity fields associated with $\rho_t$ and $\sigma_t$, respectively.
Therefore Lemma~\ref{lem:W2-differential-estimate}, applied with
\[
v_t=v[\rho_t],
\qquad
u_t=v[\sigma_t],
\]
gives, for a.e. $t\in(0,T)$ and every
$\pi_t\in\Gamma_o(\rho_t,\sigma_t)$,
\[
F'(t)
\le
\int_{\mathbb R^{d}\times\bbr^\dm}
\left\langle x-y,
v[\rho_t](x)-v[\sigma_t](y)
\right\rangle
\,d\pi_t(x,y).
\]
We decompose the right-hand side as
\[
\begin{aligned}
F'(t)
&\le
\int_{\mathbb R^{d}\times\bbr^\dm}
\left\langle x-y,
v[\rho_t](x)-v[\rho_t](y)
\right\rangle\,d\pi_t(x,y) \\
&\quad+
\int_{\mathbb R^{d}\times\bbr^\dm}
\left\langle x-y,
v[\rho_t](y)-v[\sigma_t](y)
\right\rangle\,d\pi_t(x,y) \\
&=:I_1(t)+I_2(t).
\end{aligned}
\]
We first estimate $I_1(t)$. Write
\[
A_{\rho_t}:=2c_nT_{n-1}(\Sigma(\rho_t)).
\]
Then
\[
v[\rho_t](x)-v[\rho_t](y)
=
-A_{\rho_t}(x-y).
\]
Since $\Sigma(\rho_t)$ is symmetric and positive semidefinite, we may write
\[
\Sigma(\rho_t)=P^\top D P,
\qquad
D=\operatorname{diag}(\lambda_1,\dots,\lambda_d),
\qquad
\lambda_i\ge0.
\]
Because $T_{n-1}$ is a polynomial in its matrix argument,
\[
T_{n-1}(\Sigma(\rho_t))
=
P^\top T_{n-1}(D)P.
\]
Moreover,
\[
T_{n-1}(D)
=
\operatorname{diag}
\left(
e_{n-1}(\widehat\lambda_1),
\dots,
e_{n-1}(\widehat\lambda_d)
\right),
\]
and each diagonal entry is nonnegative. Hence
$T_{n-1}(\Sigma(\rho_t))$ is symmetric nonnegative, and so is
$A_{\rho_t}$. Therefore,
\[
I_1(t)
=
-\int_{\mathbb R^{d}\times\bbr^\dm}
\left\langle x-y,A_{\rho_t}(x-y)\right\rangle
\,d\pi_t(x,y)
\le 0.
\]
We next estimate $I_2(t)$. By the Cauchy--Schwarz inequality,
\[
\begin{aligned}
I_2(t)
&\le
\left(\int_{\mathbb R^{d}\times\bbr^\dm} |x-y|^2\,d\pi_t(x,y)\right)^{1/2}
\left(
\int_{\mathbb R^{d}\times\bbr^\dm}
|v[\rho_t](y)-v[\sigma_t](y)|^2\,d\pi_t(x,y)
\right)^{1/2} \\
&=
W_2(\rho_t,\sigma_t)
\|v[\rho_t]-v[\sigma_t]\|_{L^2(\sigma_t)}.
\end{aligned}
\]
By Lemma~\ref{lem:velocity-W2-lipschitz-P2}, with $R=R_T$, there exists
$C_T>0$ such that
\[
\|v[\rho_t]-v[\sigma_t]\|_{L^2(\sigma_t)}
\le
C_TW_2(\rho_t,\sigma_t)
\]
for all $t\in[0,T]$. Hence
\[
I_2(t)
\le
C_TW_2^2(\rho_t,\sigma_t).
\]

Combining the estimates for $I_1$ and $I_2$, we obtain
\[
F'(t)
\le
C_TW_2^2(\rho_t,\sigma_t)
=
2C_TF(t)
\qquad\text{for a.e. }t\in(0,T).
\]
By Gronwall's inequality,
\[
F(t)\le e^{2C_Tt}F(0).
\]
Equivalently,
\[
W_2^2(\rho_t,\sigma_t)
\le
e^{2C_Tt}W_2^2(\rho_0,\sigma_0).
\]
Taking square roots and renaming the constant $C_T$ if necessary, we get
\[
W_2(\rho_t,\sigma_t)
\le
e^{C_Tt}W_2(\rho_0,\sigma_0),
\qquad
0\le t\le T.
\]
If $\rho_0=\sigma_0$, then the right-hand side vanishes. Therefore
\[
W_2(\rho_t,\sigma_t)=0
\qquad\text{for all }t\in[0,T],
\]
and hence
\[
\rho_t=\sigma_t
\qquad\text{for all }t\in[0,T].
\]
\end{proof}

\noindent\underline{\textbf{Proof of Theorem 1.1}}

Let $\rho_0\in\mathcal P_2(\mathbb R^d)$ be given. We first construct a candidate solution by the finite-dimensional reduction. Let \[ \Sigma(\rho_0)=P^\top D_0P, \qquad D_0=\operatorname{diag}(\lambda_1^0,\dots,\lambda_d^0), \] with $P$ orthogonal. By Proposition~\ref{prop:global-reduced-system}, the reduced eigenvalue system \[ \frac{\d}{\d t}\lambda_i(t) = -4c_n\lambda_i(t)e_{n-1}(\widehat\lambda_i(t)), \qquad 1\le i\le d, \] admits a unique global solution in $[0,\infty)^d$. Define \[ D_t:=\operatorname{diag}(\lambda_1(t),\dots,\lambda_d(t)), \qquad \Sigma_t:=P^\top D_tP. \] Using this reduced system, we define the transition map $X_t$ and the pushforward curve \[ \rho_t:=(X_t)_\#\rho_0 \] as in Definition~\ref{def:transition-pushforward}. By Proposition~\ref{prop:pushforward-recovers-moments}, the constructed curve satisfies \[ m(\rho_t)=m(\rho_0), \qquad \Sigma(\rho_t)=\Sigma_t \] for all $t\ge0$. In particular, the covariance matrix of the solution is governed by the finite-dimensional system stated in the theorem. Moreover, Proposition~\ref{prop:pushforward-weak-solution} shows that $(\rho_t)_{t\ge0}$ is a weak solution to the explicit continuity equation \[ \partial_t\rho_t+\nabla\cdot(\rho_t v[\rho_t])=0, \qquad \rho|_{t=0}=\rho_0. \] On the other hand, Theorem~\ref{thm:maximal-slope-property} shows that the same curve is a maximal slope solution for $\mathcal V_n$ with respect to $|\partial\mathcal V_n|$. More precisely, it satisfies the energy dissipation identity \[ \mathcal V_n[\rho_t] + \frac12\int_s^t |\rho'|^2(r)\,\d r + \frac12\int_s^t |\partial\mathcal V_n|^2(\rho_r)\,\d r = \mathcal V_n[\rho_s], \qquad 0\le s\le t<\infty. \] This proves the existence part. We next prove uniqueness in the PDE formulation. Let $(\sigma_t)_{t\in[0,T]}$ be any $AC^2$ weak solution to \[ \partial_t\sigma_t+\nabla\cdot(\sigma_t v[\sigma_t])=0, \qquad \sigma|_{t=0}=\rho_0 \] on a finite interval $[0,T]$. Applying Theorem~\ref{thm:W2-stability-uniqueness} to $\rho_t$ and $\sigma_t$, we obtain \[ W_2(\rho_t,\sigma_t) \le e^{C_Tt}W_2(\rho_0,\sigma_0) =0, \qquad 0\le t\le T. \] Hence $\sigma_t=\rho_t$ for all $t\in[0,T]$. Since $T>0$ is arbitrary, the $AC^2$ weak solution of the explicit continuity equation is unique globally. It remains to prove uniqueness in the variational formulation. Let $(\mu_t)_{t\in[0,T]}$ be any maximal slope solution for $\mathcal V_n$ with respect to $|\partial\mathcal V_n|$ and with initial datum $\mu_0=\rho_0$. By Corollary~\ref{cor:maximal-slope-identification-Vn}, $(\mu_t)$ is a weak solution to the same explicit continuity equation \[ \partial_t\mu_t+\nabla\cdot(\mu_t v[\mu_t])=0, \qquad \mu|_{t=0}=\rho_0. \] Since maximal slope solutions belong to $AC^2(0,T;(\mathcal P_2(\mathbb R^d),W_2))$, the uniqueness just proved for the explicit continuity equation gives \[ \mu_t=\rho_t \qquad \text{for all }t\in[0,T]. \] Again, since $T>0$ is arbitrary, the maximal slope solution is unique globally. Therefore the constructed pushforward curve is simultaneously the unique $AC^2$ weak solution of the explicit continuity equation and the unique maximal slope solution for $\mathcal V_n$ with respect to $|\partial\mathcal V_n|$.
\qed

\section{Long-time behavior and convergence rates} \label{sec:4}
\setcounter{equation}{0}

To investigate the long-time behavior of \eqref{PDE} with \eqref{initial} and \eqref{vt}, we study the long-time behavior of eigenvalues $(\lambda_1,\dots,\lambda_d)$. By Proposition \ref{prop:global-reduced-system}, the dynamics is given as
\[
\frac{\d}{\d t}\lambda_i(t) = -4\left(\frac{n+1}{n!}\right)\lambda_i(t)\,e_{n-1}(\widehat\lambda_i(t)), \qquad 1\le i\le d.
\]
From the symmetry of the above system, 
\[
\lambda_i^0<\lambda_j^0\Longrightarrow \lambda_i(t)<\lambda_j(t)\quad \forall~t>0.
\]
Without loss of generality, one can assume
\begin{align*}
\lambda_1^0\geq\lambda_2^0\geq\cdots\geq\lambda_{\dm}^0\geq0.
\end{align*}
Then, we also get
\begin{align}\label{monotoncondi}
\lambda_1(t)\geq\lambda_2(t)\geq\cdots\geq\lambda_{\dm}(t)\geq0\quad\forall~t\geq0.
\end{align}
From \eqref{2.2}, each eigenvalue $\lambda_i$ is non-negative and decreases monotonically, and therefore converges to a non-negative constant $\lambda_i^\infty$. We know at most $n-1$ numbers among $\{\lambda_1^\infty,\ldots,\lambda_d^\infty\}$ are nonzero since $\lambda_i^\infty e_{n-1}(\hat{\lambda}_i^\infty)=0$, and it implies
\[
\lambda_n^\infty=\lambda_{n+1}^\infty=\cdots=\lambda_\dm^\infty=0.
\]
The convergence rate of each $\lambda_i(t)$ depends on the initial configuration. We know there exists $\ell\in \{1, 2, \cdots, n\}$ satisfying the following condition:
\begin{align}\label{initial-condi-l}
\lambda_1^0\geq \cdots\geq \lambda_{\ell-1}^0>\lambda_{\ell}^0=\cdots=\lambda_{n-1}^0=\lambda_n^0\geq \lambda_{n+1}^0\geq\cdots \geq \lambda_d^0.
\end{align}
If this condition is satisfied, then we have
\[
\lambda_{\ell}^\infty=\lambda_{\ell+1}^\infty=\cdots=\lambda_n^\infty=\cdots=\lambda_\dm^\infty=0
\]
since $\lambda_{\ell}(t)=\lambda_{\ell+1}(t)=\cdots=\lambda_n(t)$ for all $t$ and $\lambda_n^\infty=0$.

If \(\lambda_n^0=0\), then the solution is stationary. Indeed, by the ordering
\[
\lambda_1^0\ge \cdots \ge \lambda_d^0\ge 0,
\]
we have
\[
\lambda_n^0=\lambda_{n+1}^0=\cdots=\lambda_d^0=0.
\]
Hence at most \(n-1\) eigenvalues are initially positive. For \(1\le i\le n-1\), the quantity
\(e_{n-1}(\widehat\lambda_i(0))\) vanishes, because after removing \(\lambda_i^0\) there are at most
\(n-2\) positive eigenvalues left. For \(i\ge n\), we have \(\lambda_i^0=0\). Therefore the right-hand side of the reduced system vanishes for every \(i\), and hence
\[
\lambda_i(t)=\lambda_i^0,\qquad 1\le i\le d,\quad t\ge0.
\]
Consequently, \(\Sigma(\rho_t)=\Sigma(\rho_0)\), and the transition map is the identity. Thus
\[
\rho_t=\rho_0
\qquad\text{for all }t\ge0.
\]
This proves the stationary case in Theorem~1.2. 

Now, consider the case \(n=1\) and $\lambda_n^0>0$, which is completely explicit. In this case,
\[
e_0(\widehat\lambda_i)=1,
\qquad 1\le i\le d,
\]
and the reduced eigenvalue system becomes
\[
\frac{\d}{\d t}\lambda_i(t)=-8\lambda_i(t),
\qquad 1\le i\le d.
\]
Hence
\[
\lambda_i(t)=\lambda_i^0e^{-8t},
\qquad 1\le i\le d.
\]
Equivalently, since \(T_0(\Sigma)=I\) and \(c_1=2\), the transition map satisfies
\[
\dot Q_t=-4Q_t,\qquad Q_0=I,
\]
and therefore
\[
Q_t=e^{-4t}I.
\]
Thus the solution is given explicitly by a pushforward map as follows:
\[
\rho_t=\bigl(m_0+e^{-4t}(x-m_0)\bigr)_{\#}\rho_0.
\]
Since \(\lambda_1^0>0\) in the present non-stationary case, we have
\[
\operatorname{tr}\Sigma(\rho_0)>0.
\]
Moreover, \(\rho_t\) converges in \(W_2\) to
\[
\rho_\infty=\delta_{m_0}.
\]
Indeed, since the target measure is a Dirac mass, the only coupling between \(\rho_t\) and
\(\delta_{m_0}\) gives
\[
W_2^2(\rho_t,\delta_{m_0})
=
\int_{\mathbb R^d}|x-m_0|^2\,\d\rho_t(x)
=
\operatorname{tr}\Sigma(\rho_t)
=
e^{-8t}\operatorname{tr}\Sigma(\rho_0).
\]
Consequently,
\[
W_2(\rho_t,\rho_\infty)
=
e^{-4t}\bigl(\operatorname{tr}\Sigma(\rho_0)\bigr)^{1/2}.
\]
In particular,
\[
-\ln W_2(\rho_t,\rho_\infty)
=
4t-\frac12\ln\operatorname{tr}\Sigma(\rho_0)
\simeq t.
\]
Therefore, Theorem~1.2 holds for \(n=1\), with \(\ell=1\), $\operatorname{rank}\Sigma(\rho_\infty)=0,~
\operatorname{supp}\rho_\infty=\{m_0\}.$\\

In the rest of this section, we assume $n\ge2$ and $\lambda_n^0>0$.

\begin{lemma}\label{lemma:lminus1nonzero}
Let $(\lambda_1(t), \lambda_2(t), \dots, \lambda_\dm(t))$ be a solution to \eqref{2.2} with $n\geq2$. Furthermore, the initial condition satisfies \eqref{initial-condi-l} for some $\ell\in \{1, 2, \cdots, n\}$ and $\lambda_n^0>0$. If $\ell\geq 2$, we have $\lambda_{\ell-1}^\infty\neq0$.
\end{lemma}

\begin{proof}
Since we will use the proof by contradiction, we assume $\lambda_{\ell-1}^\infty=0$. From l'H\^{o}pital's rule and $\lambda_{\ell}^\infty=0$, we have 
\[
\limsup_{t\to\infty}\frac{\lambda_{\ell-1}(t)}{\lambda_{\ell}(t)}\leq\limsup_{t\to\infty}\frac{\lambda_{\ell-1}'(t)}{\lambda_{\ell}'(t)}=\limsup_{t\to\infty}\frac{\lambda_{\ell-1}e_{n-1}(\hat{\lambda}_{\ell-1})}{\lambda_{\ell}e_{n-1}(\hat{\lambda}_{\ell})}.
\]
From the two inequalities:
\[
\lambda_{\ell-1}e_{n-1}(\hat{\lambda}_{\ell-1})\leq{d-1\choose n-1} \lambda_1\lambda_2\cdots\lambda_n\quad\text{and}\quad \lambda_{\ell}e_{n-1}(\hat{\lambda}_{\ell})\geq  \lambda_1\lambda_2\cdots\lambda_n,
\]
we get
\begin{align}\label{limsupsup}
\limsup_{t\to\infty}\frac{\lambda_{\ell-1}(t)}{\lambda_{\ell}(t)}\leq\limsup_{t\to\infty}\frac{\lambda_{\ell-1}e_{n-1}(\hat{\lambda}_{\ell-1})}{\lambda_{\ell}e_{n-1}(\hat{\lambda}_{\ell})}\leq {d-1\choose n-1}<\infty.
\end{align}
On the other hand, for $n\geq2$, we have
\begin{align*}
\frac{\d}{\d t}\ln\left(\frac{\lambda_{\ell-1}}{\lambda_{\ell}}\right)&=\frac{\d}{\d t}\left(\ln\lambda_{\ell-1}-\ln\lambda_{\ell}\right)\\
&=-4\left(\frac{n+1}{n!}\right)\left(e_{n-1}(\hat{\lambda}_{\ell-1})-e_{n-1}(\hat{\lambda}_{\ell})\right)\\
&=-4\left(\frac{n+1}{n!}\right)\left(\lambda_{\ell}e_{n-2}(\hat{\lambda}_{\ell-1},\hat{\lambda}_{\ell})-\lambda_{\ell-1}e_{n-2}(\hat{\lambda}_{\ell-1}, \hat{\lambda}_{\ell})\right)\\
&=4\left(\frac{n+1}{n!}\right)\left(\lambda_{\ell-1}-\lambda_{\ell}\right)e_{n-2}(\hat{\lambda}_{\ell-1}, \hat{\lambda}_{\ell}).
\end{align*}
We use 
\[
\left(\frac{\lambda_{\ell-1}}{\lambda_{\ell}}-1 \right)\geq \ln \left(1+\left(\frac{\lambda_{\ell-1}}{\lambda_{\ell}}-1 \right)\right)
\]
to obtain
\begin{align*}
  \frac{\d}{\d t}\ln\left(\frac{\lambda_{\ell-1}}{\lambda_{\ell}}\right)& =4\left(\frac{n+1}{n!}\right)\left(\lambda_{\ell-1}-\lambda_{\ell}\right)e_{n-2}(\hat{\lambda}_{\ell-1}, \hat{\lambda}_{\ell})\\
  &\geq4\left(\frac{n+1}{n!}\right)\lambda_{\ell}\ln\left(\frac{\lambda_{\ell-1}}{\lambda_{\ell}}\right)e_{n-2}(\hat{\lambda}_{\ell-1}, \hat{\lambda}_{\ell})\\
  &\geq4\left(\frac{n+1}{n!}\right)\lambda_{\ell}\ln\left(\frac{\lambda_{\ell-1}}{\lambda_{\ell}}\right)\frac{\prod_{j=1}^{n}\lambda_j}{\lambda_{\ell-1}\lambda_{\ell}}.
\end{align*}
This inequality can be simplified into
\[
	\frac{\d}{\d t}\ln\left(\ln\left(\frac{\lambda_{\ell-1}}{\lambda_{\ell}}\right)\right)\geq4\left(\frac{n+1}{n!}\right)\frac{1}{\lambda_{\ell-1}}\prod_{j=1}^{n}\lambda_j
\]
and we integrate both sides on $[0, t]$ to get
\begin{align*}
\ln\left(\ln\left(\frac{\lambda_{\ell-1}(t)}{\lambda_{\ell}(t)}\right)\right)-\ln\left(\ln\left(\frac{\lambda_{\ell-1}^0}{\lambda_{\ell}^0}\right)\right)\geq 4\left(\frac{n+1}{n!}\right)\int_{0}^t \frac{\prod_{j=1}^{n}\lambda_j(\tau)}{\lambda_{\ell-1}(\tau)}d\tau\quad\forall~t\geq0.
\end{align*}
By substituting \eqref{limsupsup} into the above inequality, we get
\begin{align}\label{loglogineq}
\ln\left({d-1\choose n-1}\right)-\ln\left(\ln\left(\frac{\lambda_{\ell-1}^0}{\lambda_{\ell}^0}\right)\right)\geq 4\left(\frac{n+1}{n!}\right)\int_{0}^t \frac{\prod_{j=1}^{n}\lambda_j(\tau)}{\lambda_{\ell-1}(\tau)}d\tau\quad\forall~t\geq0.
\end{align}

On the other hand, we also have the following inequality:
\begin{align}\label{logineq}
\frac{\d}{\d t}\ln\lambda_{\ell-1}=-4\left(\frac{n+1}{n!}\right)e_{n-1}(\hat{\lambda}_{\ell-1})\geq -4\left(\frac{n+1}{n!}\right){d-1\choose n-1}\frac{\prod_{j=1}^{n}\lambda_j}{\lambda_{\ell-1}}.
\end{align}
The equality follows from \eqref{2.2} and the inequality follows from
\[
\lambda_1\geq \lambda_{2}\geq\cdots\geq \lambda_n\geq0.
\]
Integrate \eqref{logineq} on $[0,t]$ to get
\[
\ln\lambda_{\ell-1}(t)-\ln\lambda_{\ell-1}^0\geq  -4\left(\frac{n+1}{n!}\right){d-1\choose n-1}\int_0^t\frac{\prod_{j=1}^{n}\lambda_j(\tau)}{\lambda_{\ell-1}(\tau)}d\tau.
\]
From the above inequality and $\lim_{t\to\infty}\lambda_{\ell-1}(t)=\lambda_{\ell-1}^\infty=0$, we can conclude that
\[
\int_{0}^\infty \frac{\prod_{j=1}^{n}\lambda_j(\tau)}{\lambda_{\ell-1}(\tau)}d\tau=\infty.
\]
However, this result contradicts   \eqref{loglogineq}. Therefore, we can conclude that $\lambda_{\ell-1}^\infty>0$.
\end{proof}

When \(\ell=1\), there is no positive limiting eigenvalue, and the products over
\(1\le j\le \ell-1\)  are understood as empty products equal to \(1\).

One can summarize the long-time behavior of eigenvalues as the following remark.
\begin{remark}
Suppose that the assumptions of Lemma \ref{lemma:lminus1nonzero} are satisfied. Then,
\[
\begin{cases}
\lambda_j^\infty>0\quad\forall~1\leq j\leq \ell-1,\\
\lambda_j^\infty=0\quad\forall~\ell\leq j\leq \dm.
\end{cases}
\]
\end{remark}

We next investigate the convergence rate of $\lambda_{\ell}(t)=\cdots=\lambda_n(t)$ first.

\begin{lemma}\label{lemma:lambdalrate}
Suppose that the assumptions of Lemma \ref{lemma:lminus1nonzero} are satisfied.

\begin{enumerate}
\item If $\ell=n$, then $\lambda_{\ell}$ converges exponentially fast to zero.

\item If $\ell<n$, then $\lambda_{\ell}(t)$ decays algebraically to zero with rate $t^{-\frac{1}{n-\ell}}$ as $t\to\infty$.
\end{enumerate}
\end{lemma}

\begin{proof}
Recall that
\[
\frac{\d}{\d t}\lambda_{\ell}=-4\left(\frac{n+1}{n!}\right)\lambda_{\ell} e_{n-1}(\hat{\lambda}_{\ell}).
\]
Its upper bound can be obtained as
\begin{align*}
\frac{\d}{\d t}\lambda_{\ell}(\tau)&=-4\left(\frac{n+1}{n!}\right)\lambda_{\ell}(\tau) e_{n-1}(\hat{\lambda}_{\ell}(\tau))\\
&\leq -4\left(\frac{n+1}{n!}\right)\lambda_1(\tau)\lambda_2(\tau)\cdots \lambda_n(\tau)\\
&\leq -4\left(\frac{n+1}{n!}\right)\lambda_1^\infty\lambda_2^\infty\cdots \lambda_{\ell-1}^\infty\lambda_{\ell}(\tau)\cdots\lambda_n(\tau).
\end{align*}
In the first inequality, we used \eqref{initial-condi-l} and only take the maximal element $\lambda_1(\tau)\lambda_2(\tau)\cdots\lambda_n(\tau)$. In the second inequality, we used the fact that each $\lambda_j$ is non-increasing and non-negative. Similarly, we also get the lower bound as follows:
\begin{align*}
\frac{\d}{\d t}\lambda_{\ell}(\tau)&=-4\left(\frac{n+1}{n!}\right)\lambda_{\ell}(\tau) e_{n-1}(\hat{\lambda}_{\ell}(\tau))\\
&\geq -4\left(\frac{n+1}{n!}\right){\dm-1\choose n-1}\lambda_1(\tau)\lambda_2(\tau)\cdots \lambda_n(\tau)\\
&\geq -4\left(\frac{n+1}{n!}\right){\dm-1\choose n-1}\lambda_1^0\lambda_2^0\cdots \lambda_{\ell-1}^0\lambda_{\ell}(\tau)\cdots\lambda_n(\tau).
\end{align*}
We replace each component in the sum ($\lambda_{\ell}(\tau) e_{n-1}(\hat{\lambda}_{\ell}(\tau))$) by the maximal element to get the first inequality. In the second inequality, we again used the fact that each $\lambda_j$ is non-increasing and non-negative.

Now, we use $\lambda_{\ell}=\lambda_{\ell+1}=\cdots=\lambda_n$ to get
\begin{align}\label{diffineq}
-C_2\lambda_{\ell}(\tau)^{n-\ell+1}\leq\frac{\d}{\d t}\lambda_{\ell}(\tau)\leq-C_1\lambda_{\ell}(\tau)^{n-\ell+1},
\end{align}
where the two positive constants $C_1$ and $C_2$ are defined as
\[
C_1:=4\left(\frac{n+1}{n!}\right)\lambda_1^\infty\cdots\lambda_{\ell-1}^\infty,\quad
C_2:=4\left(\frac{n+1}{n!}\right){\dm-1\choose n-1}\lambda_1^0\cdots\lambda_{\ell-1}^0.
\]

Now, we consider two cases.\\

\noindent\textbf{(Case 1: $\ell=n$)} For this case, \eqref{diffineq} can be simplified into
\[
-C_2\leq \frac{\d}{\d t}\ln \lambda_{\ell}\leq -C_1
\]
and it yields the exponential decay of $\lambda_{\ell}(=\lambda_n)$:
\begin{align}\label{eq:exponential}
\lambda_{\ell}^0e^{-C_2t}\leq \lambda_{\ell}(t)\leq \lambda_{\ell}^0e^{-C_1t}.
\end{align}

\noindent\textbf{(Case 2: $1\leq \ell<n$)} For this case, \eqref{diffineq} can be simplified into
\[
-C_2\leq\frac{1}{-n+\ell}\frac{\d}{\d t}\left(\lambda_{\ell}^{-n+\ell}\right)\leq -C_1
\]
and it yields an algebraic decay of $\lambda_{\ell}$ as follows:
\begin{align*}
&\left((\lambda_{\ell}^0)^{-n+\ell}+(n-\ell)C_2t\right)^{-\frac{1}{n-\ell}}
\leq\lambda_{\ell}(t)\leq \left((\lambda_{\ell}^0)^{-n+\ell}+(n-\ell)C_1t\right)^{-\frac{1}{n-\ell}},
\end{align*}
or simply:
\begin{align}\label{lambdal-alg-decay}
\lambda_{\ell}(t)\simeq t^{-\frac{1}{n-\ell}}
\qquad \text{as}\qquad t\to\infty.
\end{align}
\end{proof}

From the above lemma, we know the convergence rate of
\[
\lambda_{\ell}(t)=\lambda_{\ell+1}(t)=\cdots=\lambda_n(t)
\]
as $t\to\infty$. Since
\[
\lambda_n(t)\geq \lambda_{n+1}(t)\geq\cdots\geq\lambda_\dm(t),
\]
the components $\lambda_{n+1}(t),\dots,\lambda_\dm(t)$ converge to zero at least as fast as $\lambda_n(t)$. From this argument, we can also investigate the convergence rate of $\lambda_1(t), \lambda_2(t), \dots, \lambda_{\ell-1}(t)$ as $t\to\infty$. Recall that they converge to positive constants $\lambda_1^\infty,\lambda_2^\infty, \dots, \lambda_{\ell-1}^\infty$.

\begin{lemma}\label{rhodotestimate}
Suppose that the assumptions of Lemma \ref{lemma:lminus1nonzero} are satisfied. Then, the convergence rate of $\lambda_j$ for all $j\in\{1, 2, \dots, \ell-1\}$ can be classified as follows.

\begin{enumerate}
\item If $\ell=n$, then $\lambda_j(t)-\lambda_j^\infty$ converges exponentially fast to zero.

\item If $\ell<n$, then $\lambda_j(t)-\lambda_j^\infty$ decays algebraically to zero with rate $t^{-\frac{1}{n-\ell}}$ as $t\to\infty$.
\end{enumerate}
\end{lemma}

\begin{proof}
By a similar argument introduced in Lemma \ref{lemma:lambdalrate}, for any $j\in\{1, 2, \dots, \ell-1\}$, we also have
\begin{align}\label{diffineq-2}
-C_2\lambda_{\ell}(\tau)^{n-\ell+1}\leq\frac{\d}{\d \tau}\lambda_j(\tau)\leq-C_1\lambda_{\ell}(\tau)^{n-\ell+1}\quad\forall~\tau>0.
\end{align}

\noindent\textbf{(Case 1: $\ell=n$)} We substitute $\ell=n$ and \eqref{eq:exponential} directly into \eqref{diffineq-2} to get
\[
-C_2\lambda_{\ell}^0e^{-C_1\tau}\leq
\frac{\d}{\d \tau}\lambda_j(\tau)\leq-C_1\lambda_{\ell}^0e^{-C_2\tau}
\quad\forall~\tau>0.
\]
We integrate the above inequality on $[t,\infty)$ to get
\[
-\left(\frac{C_2}{C_1}\right)\lambda_{\ell}^0e^{-C_1t}\leq \lambda_j^\infty-\lambda_j(t)\leq -\left(\frac{C_1}{C_2}\right)\lambda_{\ell}^0e^{-C_2t}
\]
and it yields
\[
\left(\frac{C_1}{C_2}\right)\lambda_{\ell}^0e^{-C_2t}\leq \lambda_j(t)-\lambda_j^\infty\leq \left(\frac{C_2}{C_1}\right)\lambda_{\ell}^0e^{-C_1t}\quad\forall~t>0.
\]
Therefore, we know $\lambda_j(t)$ converges to $\lambda_j^\infty>0$ exponentially.\\ 

\noindent\textbf{(Case 2: $1\leq \ell<n$)} Since $\lambda_{\ell}(t)$ decays algebraically to zero with rate $t^{-\frac{1}{n-\ell}}$, there exist three positive constants $\tilde{C}_1$, $\tilde{C}_2$, and $T$ as follows:
\[
\tilde{C}_1t^{-\frac{1}{n-\ell}}\leq\lambda_{\ell}(t)\leq \tilde{C}_2t^{-\frac{1}{n-\ell}}\quad\forall~t>T.
\]
If we substitute this into \eqref{diffineq-2}, we get
\[
-C_2\tilde{C}_2^{n-\ell+1}\tau^{-\frac{1}{n-\ell}-1}\leq\frac{\d}{\d \tau}\lambda_j(\tau)\leq-C_1\tilde{C}_1^{n-\ell+1}\tau^{-\frac{1}{n-\ell}-1}
\quad\forall~\tau>T.
\]
Now, for any $t>T$, we integrate the above inequality on $[t, \infty)$ to get
\[
(n-\ell)C_1\tilde{C}_1^{n-\ell+1}t^{-\frac{1}{n-\ell}}\leq \lambda_j(t)-\lambda_j^\infty\leq (n-\ell)C_2\tilde{C}_2^{n-\ell+1}t^{-\frac{1}{n-\ell}},
\]
or simply:
\[
\lambda_j(t)-\lambda_j^\infty\simeq t^{-\frac{1}{n-\ell}}
\qquad \text{as}\qquad t\to\infty.
\]
\end{proof}

\begin{lemma}\label{lem:metric-speed-rate}
Suppose that the assumptions of Lemma \ref{lemma:lminus1nonzero} are satisfied. Then, the convergence rate of $|\rho'|(t)$ can be classified as follows.
\begin{enumerate}
\item If $\ell=n$, then $|\rho'|(t)$ converges exponentially fast to zero.

\item If $\ell<n$, then $|\rho'|(t)$ decays algebraically to zero with rate $t^{-\frac{1}{2(n-\ell)}-1}$ as $t\to\infty$.
\end{enumerate}
\end{lemma}

\begin{proof}
Since 
\[
\mathcal{V}_n[\rho_t]=\frac{n+1}{n!}e_n(\Sigma(\rho_t))=\frac{n+1}{n!}e_n(\lambda_1(\rho_t), \ldots, \lambda_\dm(\rho_t)),
\]
we have
\begin{align*}
\frac{\d}{\d t}\mathcal{V}_n[\rho_t]&=\frac{n+1}{n!}\sum_{k=1}^\dm \left(\frac{\d}{\dt}\lambda_k(t)\right)e_{n-1}\left(\hat{\lambda}_k(t)\right)\\
&=-4\left(\frac{n+1}{n!}\right)^2\sum_{k=1}^d\lambda_k(t)\left(e_{n-1}(\hat{\lambda}_k(t))\right)^2.
\end{align*}

\noindent\textbf{(Upper bound)} We can find the upper bound of the derivative as follows:
\begin{align*}
\frac{\d}{\d t}\mathcal{V}_n[\rho_t]
&=-4\left(\frac{n+1}{n!}\right)^2\sum_{k=1}^d\lambda_k(t)\left(e_{n-1}(\hat{\lambda}_k(t))\right)^2\\
&\leq -4\left(\frac{n+1}{n!}\right)^2\lambda_n(t)\left(e_{n-1}(\hat{\lambda}_n(t))\right)^2\\
&\leq-4\left(\frac{n+1}{n!}\right)^2\lambda_1(t)^2\lambda_2(t)^2\cdots\lambda_{n-1}(t)^2\lambda_n(t)\\
&\leq -4\left(\frac{n+1}{n!}\right)^2(\lambda_1^\infty)^2(\lambda_2^\infty)^2\cdots(\lambda_{\ell-1}^\infty)^2\lambda_{\ell}(t)^{2(n-\ell)+1}.
\end{align*}
In the second inequality, we just take $k=n$ case from the sum. In the third inequality, we use that the maximal element in a sum $e_{n-1}(\hat{\lambda}_n(t))$ is $\lambda_1(t)\lambda_2(t)\cdots\lambda_{n-1}(t)$. We used monotonicity of $\lambda_j$ and $\lambda_{\ell}(t)=\lambda_{\ell+1}(t)=\cdots=\lambda_n(t)$ in the last inequality.\\

\noindent\textbf{(Lower bound)} If $k\in\{1, 2, \cdots, n-1\}$, then we know $\frac{\prod_{j=1}^{n}\lambda_j}{\lambda_k}$ is the maximal term in a sum $e_{n-1}(\hat{\lambda}_k)$. Since there are ${\dm-1\choose n-1}$ terms in the sum, we get
\[
e_{n-1}(\hat{\lambda}_k)\leq {\dm-1\choose n-1}\frac{\prod_{j=1}^{n}\lambda_j}{\lambda_k}.
\]
Therefore, we get
\[
\lambda_k\left(e_{n-1}(\hat{\lambda}_k)\right)^2\leq\frac{1}{\lambda_k}{\dm-1\choose n-1}^2\left(\prod_{j=1}^{n}\lambda_j\right)^2.
\]
Since $\lambda_n\leq \lambda_k$ for this case, we have
\[
\lambda_k\left(e_{n-1}(\hat{\lambda}_k)\right)^2\leq\frac{1}{\lambda_n}{\dm-1\choose n-1}^2\left(\prod_{j=1}^{n}\lambda_j\right)^2={\dm-1\choose n-1}^2\lambda_1^2\lambda_2^2\cdots\lambda_{n-1}^2\lambda_n.
\]
Similarly, if $k\in \{n,n+1, \cdots, \dm\}$, then the maximum term in a sum $e_{n-1}(\hat{\lambda}_k)$ is $\prod_{j=1}^{n-1}\lambda_j$. Again, there are ${\dm-1\choose n-1}$ terms in the sum, we get
\[
e_{n-1}(\hat{\lambda}_k)\leq {\dm-1\choose n-1}\prod_{j=1}^{n-1}\lambda_j.
\]
For this case, we have $\lambda_n\geq \lambda_k$, and it implies
\[
\lambda_k\left(e_{n-1}(\hat{\lambda}_k)\right)^2\leq \lambda_n\left(e_{n-1}(\hat{\lambda}_k)\right)^2\leq {\dm-1\choose n-1}^2\left(\prod_{j=1}^{n-1}\lambda_j\right)^2 \lambda_n\leq {\dm-1\choose n-1}^2\lambda_1^2\lambda_2^2\cdots\lambda_{n-1}^2\lambda_n.
\]
Finally, for any $k\in\{1, 2, \cdots, \dm\}$, we can conclude that
\[
\lambda_k\left(e_{n-1}(\hat{\lambda}_k)\right)^2\leq {\dm-1\choose n-1}^2\lambda_1^2\lambda_2^2\cdots\lambda_{n-1}^2\lambda_n.
\]
From this inequality, we can find the lower bound of the derivative as follows:
\begin{align*}
\frac{\d}{\d t}\mathcal{V}_n[\rho_t]
&=-4\left(\frac{n+1}{n!}\right)^2\sum_{k=1}^d\lambda_k(t)\left(e_{n-1}(\hat{\lambda}_k(t))\right)^2\\
&\geq-4\left(\frac{n+1}{n!}\right)^2\sum_{k=1}^d{\dm-1\choose n-1}^2\lambda_1(t)^2\lambda_2(t)^2\cdots\lambda_{n-1}(t)^2\lambda_n(t)\\
&=-4\dm\left(\frac{n+1}{n!}\right)^2{\dm-1\choose n-1}^2\lambda_1(t)^2\lambda_2(t)^2\cdots\lambda_{n-1}(t)^2\lambda_n(t)\\
&\geq-4\dm\left(\frac{n+1}{n!}\right)^2{\dm-1\choose n-1}^2(\lambda_1^0)^2(\lambda_2^0)^2\cdots(\lambda_{\ell-1}^0)^2\lambda_{\ell}(t)^{2(n-\ell)+1}.
\end{align*}
In the last inequality, we again used monotonicity of $\lambda_j$ and $\lambda_{\ell}(t)=\lambda_{\ell+1}(t)=\cdots=\lambda_n(t)$.\\

From the above bounds, we get
\[
 -C_4\lambda_{\ell}(t)^{2(n-\ell)+1}\leq \frac{\d}{\d t}\mathcal{V}_n[\rho_t]\leq -C_3\lambda_{\ell}(t)^{2(n-\ell)+1}
\]
where $C_3$ and $C_4$ are positive constants:
\[
C_3:=4\left(\frac{n+1}{n!}\right)^2(\lambda_1^\infty)^2(\lambda_2^\infty)^2\cdots(\lambda_{\ell-1}^\infty)^2,\quad C_4:=4\dm\left(\frac{n+1}{n!}\right)^2{\dm-1\choose n-1}^2(\lambda_1^0)^2(\lambda_2^0)^2\cdots(\lambda_{\ell-1}^0)^2.
\]
From the following relationship:
\[
|\rho'|(t)=\left(\int_{\bbr^\dm}\|v_t\|^2\d\rho_t\right)^{1/2}=\sqrt{-\frac{\d}{\d t}\mathcal{V}_n[\rho_t]},
\]
we get
\[
C_3^{1/2}\lambda_{\ell}(t)^{n-\ell+\frac{1}{2}}\leq| \rho'|(t)\leq C_4^{1/2}\lambda_{\ell}(t)^{n-\ell+\frac{1}{2}}.
\]
By substituting the result from  Lemma \ref{lemma:lambdalrate}, we can obtain the desired result.
\end{proof}

The following lower bound is a Gelbrich-type inequality; see \cite{gelbrich1990formula}.

\begin{lemma}[Gelbrich-type lower bound]\label{lem:Gelbrich-lower-bound}
Let $\mu,\nu\in\mathcal P_2(\mathbb R^d)$, and denote their means and
covariance matrices by $m_\mu$, $m_\nu$, $\Sigma_\mu$, and $\Sigma_\nu$, respectively. Then
\[
W_2^2(\mu,\nu)
\geq |m_\mu-m_\nu|^2
+
\operatorname{tr}\Sigma_\mu
+
\operatorname{tr}\Sigma_\nu 
-
2\operatorname{tr}
\left[
\left(
\Sigma_\nu^{1/2}\Sigma_\mu\Sigma_\nu^{1/2}
\right)^{1/2}
\right].
\]
In particular, if $m_\mu=m_\nu$ and $\Sigma_\mu,\Sigma_\nu$ are simultaneously
diagonalizable, say
\[
\Sigma_\mu=P^\top\operatorname{diag}(\lambda_1,\dots,\lambda_d)P,
\qquad
\Sigma_\nu=P^\top\operatorname{diag}(\eta_1,\dots,\eta_d)P,
\]
then
\[
W_2^2(\mu,\nu)
\ge
\sum_{i=1}^d
\left(\sqrt{\lambda_i}-\sqrt{\eta_i}\right)^2.
\]
\end{lemma}

\begin{proof}
Let $\pi\in\Gamma(\mu,\nu)$ be an arbitrary transport plan, and let $(X,Y)$ be the coordinate
maps on $\mathbb R^d\times\mathbb R^d$ endowed with the probability measure
$\pi$. Set
\[
\widetilde X:=X-m_\mu,
\qquad
\widetilde Y:=Y-m_\nu.
\]
Then
\[
X-Y=(m_\mu-m_\nu)+(\widetilde X-\widetilde Y),
\]
and since
\[
\int_{\bbr^d\times \bbr^d} \widetilde X\,d\pi=0,
\qquad
\int_{\bbr^d\times \bbr^d} \widetilde Y\,d\pi=0,
\]
we obtain
\[
\begin{aligned}
\int_{\mathbb R^{d}\times\bbr^\dm} |x-y|^2\,d\pi(x,y)
&=
|m_\mu-m_\nu|^2
+
\int_{\bbr^\dm\times\bbr^\dm} |\widetilde X-\widetilde Y|^2\,d\pi \\
&=
|m_\mu-m_\nu|^2
+
\int_{\bbr^\dm\times\bbr^\dm} |\widetilde X|^2\,d\pi
+
\int_{\bbr^\dm\times\bbr^\dm} |\widetilde Y|^2\,d\pi
-
2\int_{\bbr^\dm\times\bbr^\dm} \widetilde X\cdot\widetilde Y\,d\pi \\
&=
|m_\mu-m_\nu|^2
+
\operatorname{tr}\Sigma_\mu
+
\operatorname{tr}\Sigma_\nu
-
2\operatorname{tr}C_\pi,
\end{aligned}
\]
where
\[
C_\pi:=
\int_{\mathbb R^{d}\times\bbr^\dm}
\widetilde X\widetilde Y^\top\,d\pi
\]
is the cross-covariance matrix associated with the coupling $\pi$.

It remains to estimate $\operatorname{tr}C_\pi$ from above. The block covariance
matrix
\[
\begin{pmatrix}
\Sigma_\mu & C_\pi \\
C_\pi^\top & \Sigma_\nu
\end{pmatrix}
\]
is nonnegative symmetric. Assume first that \(\Sigma_\mu\) and \(\Sigma_\nu\) are positive definite. Then there exists a matrix $R_\pi$ with
$\|R_\pi\|_{\mathrm{op}}\le1$ such that $C_\pi=\Sigma_\mu^{1/2}R_\pi\Sigma_\nu^{1/2}$ by the nonnegativity of the block covariance matrix, or equivalently by the Schur complement condition.

Therefore,
\[
\operatorname{tr}C_\pi
=
\operatorname{tr}
\left(
R_\pi\Sigma_\nu^{1/2}\Sigma_\mu^{1/2}
\right).
\]
By the duality between the operator norm and the trace norm,
\[
\operatorname{tr}
\left(
R_\pi\Sigma_\nu^{1/2}\Sigma_\mu^{1/2}
\right)
\le
\left\|
\Sigma_\nu^{1/2}\Sigma_\mu^{1/2}
\right\|_*.
\]
Moreover,
\[
\left\|
\Sigma_\nu^{1/2}\Sigma_\mu^{1/2}
\right\|_*
=
\operatorname{tr}
\left[
\left(
\Sigma_\nu^{1/2}\Sigma_\mu\Sigma_\nu^{1/2}
\right)^{1/2}
\right].
\]
Hence
\[
\operatorname{tr}C_\pi
\le
\operatorname{tr}
\left[
\left(
\Sigma_\nu^{1/2}\Sigma_\mu\Sigma_\nu^{1/2}
\right)^{1/2}
\right].
\]
If $\Sigma_\mu$ or $\Sigma_\nu$ is singular, the same estimate follows by
replacing them with $\Sigma_\mu+\varepsilon I$ and
$\Sigma_\nu+\varepsilon I$, applying the previous argument, and then letting
$\varepsilon\downarrow0$.

Consequently, for every coupling $\pi\in\Gamma(\mu,\nu)$,
\[
\begin{aligned}
\int_{\mathbb R^{d}\times\bbr^\dm} |x-y|^2\,d\pi(x,y)
\ge\;&
|m_\mu-m_\nu|^2
+
\operatorname{tr}\Sigma_\mu
+
\operatorname{tr}\Sigma_\nu
-
2\operatorname{tr}
\left[
\left(
\Sigma_\nu^{1/2}\Sigma_\mu\Sigma_\nu^{1/2}
\right)^{1/2}
\right].
\end{aligned}
\]
Taking the infimum over all $\pi\in\Gamma(\mu,\nu)$ gives the desired lower
bound for $W_2^2(\mu,\nu)$.

Finally, suppose that $m_\mu=m_\nu$ and that $\Sigma_\mu,\Sigma_\nu$ are
simultaneously diagonalizable:
\[
\Sigma_\mu=P^\top\operatorname{diag}(\lambda_1,\dots,\lambda_d)P,
\qquad
\Sigma_\nu=P^\top\operatorname{diag}(\eta_1,\dots,\eta_d)P.
\]
Then
\[
\operatorname{tr}\Sigma_\mu+\operatorname{tr}\Sigma_\nu
=
\sum_{i=1}^d(\lambda_i+\eta_i),
\]
and
\[
\operatorname{tr}
\left[
\left(
\Sigma_\nu^{1/2}\Sigma_\mu\Sigma_\nu^{1/2}
\right)^{1/2}
\right]
=
\sum_{i=1}^d\sqrt{\lambda_i\eta_i}.
\]
Therefore,
\[
W_2^2(\mu,\nu)
\ge
\sum_{i=1}^d
\left(
\lambda_i+\eta_i-2\sqrt{\lambda_i\eta_i}
\right)
=
\sum_{i=1}^d
\left(\sqrt{\lambda_i}-\sqrt{\eta_i}\right)^2.
\]
This proves the claim.
\end{proof}

\noindent\underline{\textbf{Proof of Theorem 1.2}}

The stationary case \(\lambda_n^0=0\) and the case \(n=1\) have already been treated
at the beginning of this section. Hence we assume throughout the rest of the proof that
\(n\ge2\) and \(\lambda_n^0>0\).

By Lemma~\ref{lem:metric-speed-rate}, the metric derivative \(|\rho'|(t)\) is integrable
on \([T,\infty)\) for some sufficiently large \(T>0\). Indeed, if \(\ell=n\), then
\(|\rho'|(t)\) decays exponentially, while if \(\ell<n\), then
\[
|\rho'|(t)\lesssim t^{-1-\frac{1}{2(n-\ell)}}.
\]
Therefore
\[
\int_T^\infty |\rho'|(\tau)\,d\tau<\infty.
\]
For \(s>t\ge T\), the absolute continuity of \((\rho_t)\) in \(W_2\) gives
\[
W_2(\rho_t,\rho_s)
\le
\int_t^s |\rho'|(\tau)\,d\tau.
\]
Letting \(t,s\to\infty\), we see that \((\rho_t)_{t\ge0}\) is a Cauchy curve in
\((\mathcal P_2(\mathbb R^d),W_2)\). Since \((\mathcal P_2(\mathbb R^d),W_2)\) is complete,
there exists \(\rho_\infty\in\mathcal P_2(\mathbb R^d)\) such that
\[
W_2(\rho_t,\rho_\infty)\to0
\qquad\text{as }t\to\infty.
\]
Moreover, by letting \(s\to\infty\) in the preceding estimate, we obtain
\[
W_2(\rho_t,\rho_\infty)
\le
\int_t^\infty |\rho'|(\tau)\,d\tau.
\]

We next identify the covariance of the limiting measure. Since \(\rho_t\to\rho_\infty\)
in \(W_2\), the means and covariance matrices converge. Thus
\[
m(\rho_\infty)=m_0,
\qquad
\Sigma(\rho_\infty)
=
P^\top
\operatorname{diag}(\lambda_1^\infty,\dots,\lambda_d^\infty)P.
\]
By Remark~4.1,
\[
\lambda_j^\infty>0 \quad\text{for }1\le j\le \ell-1,
\qquad
\lambda_j^\infty=0 \quad\text{for }j\ge \ell.
\]
Therefore
\[
\operatorname{rank}\Sigma(\rho_\infty)=\ell-1.
\]

Let \(w_1,\dots,w_d\) be the orthonormal eigenvectors of \(\Sigma(\rho_0)\) associated with
\(\lambda_1^0,\dots,\lambda_d^0\). Since the eigenspaces are preserved by the covariance
flow, we have
\[
\ker\Sigma(\rho_\infty)
=
\operatorname{span}\{w_\ell,\dots,w_d\}.
\]
For \(z\in \ker\Sigma(\rho_\infty)\), we have
\[
0
=
z^\top\Sigma(\rho_\infty)z
=
\int_{\mathbb R^d}
|z\cdot(x-m_0)|^2\,d\rho_\infty(x).
\]
Hence \(z\cdot(x-m_0)=0\) for \(\rho_\infty\)-a.e. \(x\). Taking
\(z=w_\ell,\dots,w_d\), we obtain
\[
\operatorname{supp}\rho_\infty
\subset
m_0+\operatorname{span}\{w_1,\dots,w_{\ell-1}\}.
\]

We now prove the convergence rates. From the estimate
\[
W_2(\rho_t,\rho_\infty)
\le
\int_t^\infty |\rho'|(\tau)\,d\tau
\]
and Lemma~\ref{lem:metric-speed-rate}, we obtain
\[
\begin{cases}
W_2(\rho_t,\rho_\infty)\lesssim e^{-ct}
\quad\text{for some }c>0, & n=\ell,\\[1mm]
W_2(\rho_t,\rho_\infty)
\lesssim
t^{-\frac{1}{2(n-\ell)}}, & n>\ell.
\end{cases}
\]
In particular, in the case \(n=\ell\), this implies
\[
-\ln\left(W_2(\rho_t,\rho_\infty)\right)\gtrsim t.
\]

For the lower bound, we use Lemma~\ref{lem:Gelbrich-lower-bound}. Since
\(m(\rho_t)=m(\rho_\infty)=m_0\) and \(\Sigma(\rho_t)\), \(\Sigma(\rho_\infty)\) are
diagonalized by the same orthogonal matrix, we have
\[
W_2^2(\rho_t,\rho_\infty)
\ge
\sum_{i=1}^d
\left(\sqrt{\lambda_i(t)}-\sqrt{\lambda_i^\infty}\right)^2.
\]
Since \(\lambda_i^\infty=0\) for \(i\ge\ell\), it follows that
\[
W_2^2(\rho_t,\rho_\infty)
\ge
\sum_{i=\ell}^d \lambda_i(t)
\ge
\lambda_\ell(t).
\]
Hence
\[
W_2(\rho_t,\rho_\infty)
\ge
\lambda_\ell(t)^{1/2}.
\]
Using Lemma~\ref{lemma:lambdalrate}, we get
\[
\begin{cases}
W_2(\rho_t,\rho_\infty)\gtrsim e^{-Ct}
\quad\text{for some }C>0, & n=\ell,\\[1mm]
W_2(\rho_t,\rho_\infty)
\gtrsim
t^{-\frac{1}{2(n-\ell)}}, & n>\ell.
\end{cases}
\]
In particular, in the case \(n=\ell\), this implies
\[
-\ln\left(W_2(\rho_t,\rho_\infty)\right)\lesssim t.
\]
Combining the upper and lower bounds, we conclude that
\[
\begin{cases}
-\ln\left(W_2(\rho_t,\rho_\infty)\right)\simeq t, & n=\ell,\\[1mm]
W_2(\rho_t,\rho_\infty)\simeq t^{-\frac{1}{2(n-\ell)}}, & n>\ell.
\end{cases}
\]
This completes the proof.
\qed

 \bibliographystyle{plain}
 \bibliography{lit}
\end{document}